\documentclass[12pt,draft,reqno]{amsart}

\usepackage{amssymb} 
\usepackage{dsfont} 
\usepackage{mathrsfs} 
\usepackage{stmaryrd} 
\usepackage{color}
\usepackage{xcolor}

\DeclareMathAlphabet{\mathpzc}{OT1}{pzc}{m}{it} 

\newtheorem{Thm}{Theorem}[section]
\newtheorem{Cor}{Corollary}[section]
\newtheorem{Lem}{Lemma}[section]
\newtheorem{Prop}{Proposition}[section]

\newtheorem{Def}{Definition}[section]

\theoremstyle{definition}
\newtheorem{Rem}{Remark}[section]

\theoremstyle{definition}

\makeatletter
\@addtoreset{equation}{section}
\makeatother

\newcommand\function{\longrightarrow} 
\newcommand\setmeno{\!\smallsetminus\!} 
\newcommand\en{\mathbb{N}} 
\newcommand\ar{\mathbb{R}} 
\renewcommand\H{\mathcal{H}} 
\newcommand\duality[2]{\langle #1,#2 \rangle} 
\newcommand\lduality[2]{\left\langle #1,#2 \right\rangle} 
\providecommand{\clint}[1]{\hspace{0.045ex}\left[#1\right]} 
\providecommand{\opint}[1]{\hspace{0.15ex}\left]#1\right[\hspace{0.15ex}} 
\providecommand{\clsxint}[1]{\hspace{0.1ex}\left[#1\right[\hspace{0.15ex}} 
\providecommand{\cldxint}[1]{\hspace{0.15ex}\left]#1\right]} 
\newcommand\norm[2]{\Vert #1\Vert_{#2}} 
\newcommand\lnorm[2]{\left\Vert #1\right\Vert_{#2}} 
\renewcommand\d{\textsl{d}} 
\newcommand\K{\mathcal{K}} 
\DeclareMathOperator{\Proj}{Proj} 
\newcommand\Z{\mathcal{Z}} 
\newcommand\void{\varnothing} 
\newcommand{\Conv}{\mathscr{C}} 
\renewcommand{\P}{\mathrm{P}} 
\newcommand\A{\mathcal{A}} 
\newcommand\B{\mathcal{B}} 
\newcommand\e{\textsl{e}} 
\newcommand\parti{\mathscr{P}}
\DeclareMathOperator{\ret}{R} 
\newcommand{\BR}{{\textsl{BR}\hspace{0.17ex}}} 
\newcommand\loc{\mathrm{loc}} 

\newcommand{\Sob}{{\textsl{W}\hspace{0.17ex}}} 

\DeclareMathOperator{\V}{V} 

\newcommand\indicator{\mathds{1}} 

\newcommand{\eps}{\varepsilon} 

\newcommand{\borel}{\mathscr{B}} 
\renewcommand{\L}{{\textsl{L}\hspace{0.17ex}}} 
\newcommand\leb{\mathpzc{L}} 
\DeclareMathOperator{\de}{d \! \hspace{0.2ex}} 
\newcommand{\Step}{{\textsl{St}\hspace{0.17ex}}} 

\DeclareMathOperator{\Int}{int} 

\DeclareMathOperator{\cont}{Cont}  
\DeclareMathOperator{\discont}{Discont}  
\newcommand{\Czero}{{\textsl{C}\hspace{0.18ex}}} 
\DeclareMathOperator{\Lipcost}{Lip} 
\newcommand{\Lip}{{\textsl{Lip}\hspace{0.15ex}}} 

\newcommand\C{\mathcal{C}} 
\newcommand\G{\mathcal{G}} 

\newcommand{\BV}{{\textsl{BV}\hspace{0.17ex}}} 

\DeclareMathOperator{\pV}{V} 
\renewcommand\r{\textsl{r}} 
\newcommand\vartot[1]{\!\left\bracevert\! #1 \!\right\bracevert\!} 
\DeclareMathOperator{\D}{D\!} 

\newcommand{\Ctilde}{\widetilde{\mathcal{C}}} 

\newcommand{\yhat}{\widehat{y}\,} 

\newcommand{\Sw}{{\mathsf{S}}} 

\renewcommand\sp{\hspace{3.1ex}} 

\definecolor{blu}{rgb}{0.1,0.1,1}

\definecolor{green}{rgb}{0.0, 0.5, 0.0}

\definecolor{marr}{rgb}{0.63, 0.47, 0.35}

\begin{document}


\title[Sweeping processes]{Prox-regular sweeping processes \\ with bounded retraction}

\author{Vincenzo Recupero}
\thanks{The author is a member of GNAMPA-INdAM}


\address{\textbf{Vincenzo Recupero} \\
        Dipartimento di Scienze Matematiche \\ 
        Politecnico di Torino \\
        C.so Duca degli Abruzzi 24 \\ 
        I-10129 Torino \\ 
        Italy. \newline
        {\rm E-mail address:}
        {\tt vincenzo.recupero@polito.it}}




\begin{abstract}
The aim of this paper is twofold. On one hand we prove that the Moreau's sweeping processes driven by a uniformly prox-regular moving set with local bounded retraction have a unique solution provided that the coefficient of prox-regularity is larger than the size of any jump of the driving set. On the other hand we show how the case of local bounded retraction can be easily reduced to the $1$-Lipschitz continuous case: indeed we first solve the Lipschitz continuous case by means of the so called ``catching-up algorithm", and we reduce the local bounded retraction case to the Lipschitz one by using a reparametrization technique for functions with values in the family of prox-regular sets. 
\end{abstract}


\maketitle


\thispagestyle{empty}


\section{Introduction}

A \emph{sweeping process} is an evolution problem with unilateral constraints which was originally introduced by J.J. Moreau in \cite{Mor71} and that can be described as follows. Let $\H$ be a real Hilbert space, and for every 
$t \in \clint{0,T}$, $T > 0$, let $\C(t)$ be a given nonempty, closed subset of $\H$, so that $\C(t)$ can be considered as a moving set whose shape can vary during the time evolution. One then has to find a function 
$y : \clint{0,T} \function \H$ 
such that
\begin{alignat}{3}
  & y(t) \in \C(t) & \quad & \forall t \in \clint{0,T}, \label{y in C - Lip - intro} \\
  & y'(t) \in -N_{\C(t)}(y(t)) & \quad & \text{for $\leb^{1}$-a.e. $t \in \clint{0,T}$}, \label{diff. incl. - Lip - intro} \\ 
  & y(0) = y_{0}, &  \label{in. cond. - Lip - intro}  
\end{alignat}
$y_0$ being a prescribed point in $\C(0)$. Here $\leb^1$ is the Lebesgue measure, and $N_{\C(t)}(y(t))$ is the (proximal) exterior normal cone to $\C(t)$ at $y(t)$. In the original setting of Moreau's paper \cite{Mor71}, the sets 
$\C(t)$ are assumed to be convex, and the mapping $t \longmapsto \C(t)$ is supposed to be Lipschitz continuous in time, when the family of closed subsets of $\H$ is endowed with the Hausdorff metric. There, the sweeping process \eqref{y in C - Lip - intro}--\eqref{in. cond. - Lip - intro} is solved by approximating the problem by what is now called the Moreau-Yosida regularization. He also solves \eqref{y in C - Lip - intro}--\eqref{in. cond. - Lip - intro} with a more general absolutely continuous driving set $\C(t)$ via an arc-length reparametrization of $\C(t)$: indeed if $\phi(t)$ is the variation of $\C$ on $\clint{0,t}$, then one finds a Lipschitz moving set $\Ctilde$ such that 
\[
  \C(t)= \Ctilde(\phi(t))
\] 
and the problem can be easily reduced to the Lipschitz continuous case by showing that if $\yhat(\tau)$ is the Lipschitz continuous solution of the sweeping processes driven by $\Ctilde(\tau)$, then $y(t) = \yhat(\phi(t))$ is the solution of the process driven by $\C(t)$.

This analysis of sweeping processes in the convex setting continued in  \cite{Mor72, Mor73, Mor74} and culminated in \cite{Mor77} where right-continuous driving sets $\C(t)$ with bounded variation are considered, and in this case problem \eqref{y in C - Lip - intro}--\eqref{in. cond. - Lip - intro} has to be reformulated by looking for a right-continuous solution $y$ with bounded variation such that the distributional derivative of $y$ can be written as 
$\D y = v \mu$ for a suitable positive measure $\mu$ in $\clint{0,T}$ and a density function $v \in \L^1(\mu;\H)$ for which \eqref{y in C - Lip - intro} and \eqref{in. cond. - Lip - intro} hold together with the inclusion
\begin{equation}\label{diff.incl.-BV-intro}
  v(t) \in -N_{\C(t)}(y(t)) \qquad \text{for $\mu$-a.e. $t \in \clint{0,T}$},
\end{equation}
which is the $\BV$ counterpart of \eqref{diff. incl. - Lip - intro}. Problem \eqref{diff.incl.-BV-intro} was solved in \cite{Mor77} by means of an implicit time discretization scheme which is usually called 
\emph{catching-up algorithm}. Actually in \cite{Mor77} the moving set $\C$ is assumed to be only with bounded right continuous retraction, rather than with bounded variation, i.e. the Hausdorff distance is replaced by the asymmetric distance $\e(\A,\B) := \sup_{x \in \A} \d(x,\B)$, called \emph{excess of $\A$ over $\B$}, with $\A$, $\B$ nonempty closed convex sets in $\H$, and it is assumed that $\sup \sum \e(\C(t_{j-1}), \C(t_j)) < \infty$, where the supremum is taken over all subdivisions of $0 = t_0 < \cdots < t_m = T$ (all the precise definitions will be provided in the next section).

Let us mention that the $\BV$ problem can also be solved by an arc-length reparametrization technique as done in \cite{Rec08, Rec11b} for the continuous case and in \cite{Rec09b, Rec16a, Rec20} for the $\BV$ discontinuous case. In this last case the arc-length reparametrization must be completed on the jump intervals by means of a suitable geodesic in the space of closed convex sets endowed with the Hausdorff metric. Convex sweeping processes in the $BV$ framework have received a great deal of attention and several results can be found, e.g., in the monograph \cite{Mon93}, in papers 
\cite{Cas73, Cas76, Cas83, CasDucVal93, BroKreSch04, BroThi10, KreRoc11, AdlHadThi14, Rec15a, Thi16, DimMauSan16, KopRec16, RecSan18}, and in the references therein.

It is very natural to try to relax the convexity assumption on the driving set $\C(t)$, and the first papers where the non-convex case is addressed are \cite{Val88, CasMon96, ColGon99, Ben00}. In the first years of this century, the analysis of non-convex sweeping processes started concentrating on the notion of \emph{uniformly prox-regular} sets, which can be described as those closed subsets $\K$ for which there exists $r > 0$ such that on the neighborhood $U_r(\K) := \{x \in \H\ :\ d(x, \K) < r\}$ the projection on $\K$ is a singleton and is continuous,
so that such sets can be regarded as a natural generalization of convex sets. Uniformly prox-regular sets were introduced by Federer in \cite{Fed59} (as \emph{positively reached sets}) in the finite dimensional case. Further properties of these sets were investigated by Vial in \cite{Via83} (where the term \emph{weak convex set} is used). The notion of prox-regularity was later extended to infinite dimensional spaces by Clarke, Stern and Wolenski in \cite{ClaSteWol95}, and finally the local properties where investigated by Poliquin, Rockafellar and Thibault in \cite{PolRocThi00}.

Existence and uniqueness results for sweeping processes in the prox-regular setting were achieved by Colombo and Monteiro Marques in \cite{ColMon03} and by Thibault in \cite{Thi03}, culminating in the paper \cite{EdmThi06} by Edmond and Thibault. Since then, prox-regular sweeping processes has called the attention of many other authors: we mention, for instance, 
\cite{BouThi05, CheMon07, BerVen12, SenThi14, BerVen15, AdlNacThi17, KMR21, KMR22, KMR23} and the general result \cite{NacThi}. 

In the above mentioned literature on non-convex sweeping processes, the continuity assumptions on the moving driving set are always related to the Hausdorff distance and not to the retraction. To the best of our knowledge, only in \cite{Thi16} it is proved the uniqueness of solutions to prox-regular sweeping processes with bounded retraction, but the existence is proved under the assumption that $\C(t)$ is convex and with bounded retraction (and more generally with bounded truncated retraction).

In the present paper we address the case of prox-regular sweeping processes when the driving moving set has locally bounded retraction. We first solve the Lipschitz continuous case by means of the so called ``catching-up algorithm", and we reduce the general case of locally bounded retraction to the Lipschitz one by using the reparametrization technique of \cite{Rec20} adapted to functions with values in the family of prox-regular sets.

The plan of the paper is the following. In Section \ref{S:prelim} we present the preliminaries needed to state the problems and the theorems. In Section \ref{S:Lip case} we solve the prox-regular sweeping processes with local bounded retraction when $\C(t)$ is $1$-Lipschitz continuous with respect to the excess. Finally in Section \ref{proofs} we solve the general prox-regular processes with locally bounded retraction by reducing them to the Lipschitz case thanks to the reparametrizion method cited above.


\section{Preliminaries and notations}\label{S:prelim}

In this section we recall the main definitions and tools needed in the paper. The set of integers greater than or equal to $1$ will be denoted by $\en$. 

\subsection{Prox-regular sets}

We recall here some notions of non-smooth analysis. We refer the reader to the monographs 
\cite{RocWet98, ClaLedSteWol98, Thi23a, Thi23b}. Throughout this paper we assume that
\begin{equation}\label{H-prel}
\begin{cases}
  \text{$\H$ is a real Hilbert space with inner product $(x,y) \longmapsto \duality{x}{y}$}, \\
  \norm{x}{} := \duality{x}{x}^{1/2}, \quad x \in \H,
\end{cases}
\end{equation}
and we endow $\H$ with the natural metric defined by $(x,y) \longmapsto \norm{x-y}{}$, $x, y \in \H$.
If $\rho > 0$ and $x \in \H$ we set $B_\rho(x) := \{y \in \H\ :\ \norm{y-x}{} < \rho\}$ and 
$\overline{B}_\rho(x) :=  \{y \in \H\ :\ \norm{y-x}{} \le \rho\}$. If $\Z \subseteq \H$, the closure and the boundary of $\Z$ will be respectively denoted by $\overline{\Z}$ and $\partial \Z$. If $x \in \H$ we also set 
$\d_\Z(x) := \d(x,\Z) := \inf_{s \in \Z} \norm{x-s}{}$, defining in this way a function $\d_\Z : \H \function \mathbb{R}$.
Let us recall that if $U$ is an open subset of $\H$ and if $f : U \function \mathbb{R}$, then we say that 
$f$ is \emph{(Frech\'et) differentiable at $x \in U$} if there exists a (unique) $\nabla f(x) \in \H$ such that 
$\lim_{h \to 0} (f(x+h) - f(x) - \duality{\nabla f(x)}{h})/\norm{h}{} = 0$.

\begin{Def}
If $\Z \subseteq \H$, $\Z \neq \void$, and $y \in \H$ then we set
\begin{equation}
 \Proj_\Z(y) := \left\{x \in \Z\ :\ \norm{x-y}{} = \inf_{z \in \Z} \norm{z-y}{}\right\}.
\end{equation}
\end{Def}

Statement (2.3) below can be found, e.g., in \cite[Proposition 1.3]{ClaLedSteWol98}. In fact, given 
$y \in \mathcal{H}$ and $x \in \Z$, we see that $\|x-y\|^2 \leq \|z-y\|^2$ (resp. $\|x-y\|^2 < \|z-y\|^2$) for all $z \in \Z \setmeno \{x\}$ if and only if 
$x \in \Proj_\Z(y)$ (resp. $\{x\} = \Proj_\Z(y)$), and writing
\[
  \|z-y\|^2 = \|z-x\|^2 - 2\langle y-x,z-x\rangle + \|x-y\|^2
\]
we then have the following

\begin{Prop}
If $\Z \subseteq \H$, $\Z \neq \void$, and $y \in \H$, then 
\begin{equation}
 x \in \Proj_\Z(y) \ \Longleftrightarrow \ 
 \left[\ x \in \Z,\quad  \duality{y-x}{z-x} \le \frac{1}{2}\norm{z-x}{}^2 \ \ \forall z \in \Z\ \right] \label{ch-pr} 
\end{equation}
and
\begin{equation}
 \{x\} = \Proj_\Z(y) \quad \Longleftrightarrow \ 
 \left[\ x \in \Z,\quad \duality{y-x}{z-x} < \frac{1}{2}\norm{z-x}{}^2 \ \ \forall z \in \Z \setmeno \{x\}\ \right] \label{ch-pr-str}
\end{equation}
\end{Prop}

If $\K$ is a nonempty closed subset of $\H$ and $x \in \K$, then $N_\K(x)$ denotes the \emph{(exterior proximal) normal cone of $\K$ at $x$} which is defined by setting 
\begin{equation}\label{normal cone}
  N_\K(x) := \{\lambda(y-x) \ :\ x \in \Proj_{\K}(y),\ y \in \H,\ \lambda \ge 0\}. 
\end{equation}

Let us also recall the following proposition (see, e.g., \cite[Proposition 1.5]{ClaLedSteWol98}).

\begin{Prop}\label{propsigma}
Assume that $\K$ is a closed subset of $\H$ and that $x \in \K$. We have that $u \in N_\Z(x)$ if and only if there exists $\sigma \ge 0$ such that 
\begin{equation}\label{prox normal ineq}
  \duality{u}{z-x} \le \sigma \norm{z-x}{}^2 \qquad \forall z \in \K.
\end{equation}
\end{Prop}

Now we recall the notion of \emph{uniformly prox-regular set}, introduced in 
\cite[Section 4, Theorem 4.1-(d)]{ClaSteWol95}) under the name of \emph{proximal smooth set}, and generalized in \cite[Definition 1.1, Definition 2.4, Theorem 4.1]{PolRocThi00}.

\begin{Def}
If $\K$ is a nonempty closed subset of $\H$ and if $r \in \opint{0,\infty}$, then we say that $\K$ is 
\emph{$r$-prox-regular} if for every $y \in \{v \in \H\ :\ 0 < \d_{\K}(v) < r\}$ we have that $\Proj_\K(y) \neq \void$ and
\[
  x \in \Proj_\K\left(x+r\frac{y-x}{\norm{y-x}{}}\right) \qquad \forall x \in \Proj_\K(y).
\]
The family of $r$-prox-regular subsets of $\H$ will be denoted by $\Conv_r(\H)$. We will also indicate by 
$\Conv_0(\H)$ the family of nonempty closed subsets of $\H$.
\end{Def}

It is useful to take into account the following easy property.

\begin{Prop}\label{1proj in seg}
Let $\K$ be a closed subset of $\H$. If $y \in \H \setmeno \K$ and $x \in \Proj_\K(y)$ then 
$\Proj_\K(x + t(y-x)) = \{x\}$ for every $t \in \opint{0,1}$. Hence if $\K$ is also $r$-prox-regular for some $r > 0$ it follows that $\Proj_\K(u)$ is a singleton for every $u \in \H$ such that $\d_{\K}(u) < r$.
\end{Prop}

\begin{proof}
If $t \in \opint{0,1}$ then, thanks to \eqref{ch-pr}, for every $z \in \K \setmeno \{x\}$ we have that 
\[
  \duality{x + t(y-x) - x}{z - x} = t\duality{y-x}{z-x} \le \frac{t}{2}\norm{z-x}{}^2 < \frac{1}{2}\norm{z-x}{}^2,
\]
therefore $\Proj_\K(x + t(y-x)) = \{x\}$ by virtue of \eqref{ch-pr-str} and the first statement is proved. If $\K$ is also $r$-prox-regular and $\d_{\K}(u) < r$, pick $x \in \Proj_\K(u)$ and let $y = x + r(u-x)/\norm{u-x}{}$. Then 
$x \in \Proj_{K}(y)$, $u = x + t_u(y-x)$ with $t_u = r/\norm{u-x}{} \in \opint{0,1}$, so that $\{x\} = \Proj_\K(u)$. 
\end{proof}

\begin{Def}
If $r \in \opint{0,\infty}$ and if $\K \subseteq \H$ is $r$-prox-regular, then we can define the function
$\P_\K : \{y \in \H\ :\ \d_\K(y) < r\} \function \K$ by setting $\P_\K(y) := x$ where $\{x\} = \Proj_\K(y)$ for every 
$y \in \H$ such that $\d_\K(y) < r$.
\end{Def}

The following characterizations of prox-regularity are very useful. The proofs can be found in 
\cite[Theorem 4.1]{PolRocThi00} and in \cite[Theorem 16]{ColThi10}.

\begin{Thm}\label{charact proxreg}
Let $\K$ be a nonempty closed subset of $\H$ and let $r \in \opint{0,\infty}$. The following 
statements are equivalent.
\begin{itemize}
\item[(i)] $\K$ is $r$-prox-regular.
\item[(ii)] $\d_\K$ is differentiable in $\{y \in \H\ :\ 0 < \d_\K(y) < r\}$.
\item[(iii)] For every $x \in \K$ and $n \in N_\K(x)$ we have 
\[
  \duality{n}{z-x} \le \frac{\norm{n}{}}{2r}\norm{z-x}{}^2 \qquad \forall z \in \K.
\]
\end{itemize}
\end{Thm}

\subsection{Functions of bounded variation}

Let $I$ be an interval of $\ar$. The set of $\H$-valued continuous functions defined on $I$ is denoted by 
$\Czero(I;\H)$. For a function $f : I \function \H$, the continuity set of a function $f : I \function \H$ is denoted by 
$\cont(f)$, while $\discont(f) := \H \setmeno \cont(f)$. For $S \subseteq I$ we write 
$\Lipcost[f,S] := \sup\{\norm{f(t)-f(s)}{}/|t-s|\ :\ s, t \in S,\ s \neq t\}$, $\Lipcost[f] := \Lipcost[f,I]$, the Lipschitz constant of $f$, and $\Lip(I;\H) := \{f : I \function \H\ :\ \Lipcost[f] < \infty\}$, the set of $\H$-valued Lipschitz continuous functions on $I$. As usual 
$\Lip_\loc(I;\H) := \{f : I \function \H\  :\ \Lipcost[f,J] < \infty \ \text{$\forall J$  compact in $I$}\}$.
If $f : I \function \ar$ and if $(f(t) - f(s))(t-s) \ge 0$ (respectively 
$(f(t) - f(s))(t-s) > 0$) we say that $f$ is \emph{increasing} (respectively \emph{strictly increasing}).

\begin{Def}
Given an interval $I \subseteq \ar$, a function $f : I \function \H$, and a subinterval $J \subseteq I$, the \emph{variation of $f$ on $J$} is defined by
\begin{equation}\notag
  \pV(f,J) := 
  \sup\left\{
           \sum_{j=1}^{m} \norm{(f(t_{j-1})- f(t_{j})}{}\ :\ m \in \en,\ t_{j} \in J\ \forall j,\ t_{0} < \cdots < t_{m} 
         \right\}.
\end{equation}
If $\pV(f,I) < \infty$ we say that \emph{$f$ is of bounded variation on $I$} and we set 
\[
  \BV(I;\H) := \{f : I \function \H\ :\ \pV(f,I) < \infty\} 
\]
and $\BV_\loc(I;\H) := \{f : I \function \H\ :\ \pV(f,J) < \infty\ \text{$\forall J$ \emph{compact} in $I$}\}$.
\end{Def}

It is well known that the completeness of $\H$ implies that every $f \in \BV(I;\H)$ admits one sided limits 
$f(t-), f(t+)$ at every point $t \in I$, with the convention that $f(\inf I-) := f(\inf I)$ if $\inf I \in I$, and that 
$f(\sup I+) := f(\sup I)$ if $\sup I \in I$. Moreover $\discont(f)$ is at most countable. We set 
\[
  \BV^\r(I;\H) := \{f \in \BV(I;\H)\ :\ f(t) = f(t+) \quad \forall t \in I\},
\]
\[
 \BV_\loc^\r(I;\H) := \{f \in \BV_\loc(I;\H)\ :\ f(t) = f(t+) \quad \forall t \in I\}, 
\]
$\Czero\BV(I;\H) := \BV(I;\H) \cap \Czero(I;\H)$, and $\Czero\BV_\loc(I;\H) := \BV_\loc(I;\H) \cap \Czero(I;\H)$. 
We have $\Lip_\loc(I;\H) \subseteq \BV_\loc(I;\H)$.

\subsection{Differential measures}\label{differential measures}

Given an interval $I$ of the real line $\mathbb{R}$, the family of Borel sets in $I$ is denoted by $\borel(I)$. If 
$\mu : \borel(I) \function \clint{0,\infty}$ is a measure, $p \in \clsxint{1,\infty}$, then the space of $\H$-valued functions which are $p$-integrable with respect to $\mu$ will be denoted by $\L^p(I, \mu; \H)$ or simply by 
$\L^p(\mu; \H)$. By $\L^\infty(I, \mu; B)$, or $\L^\infty(\mu; B)$, we denote the space of  maps $f : I \function \H$ such that there exists a $\mu$-measurable $g : I \function \H$ for which $f = g$ $\mu$-almost everywhere on $I$. For the theory of integration of vector valued functions we refer, e.g., to \cite[Chapter VI]{Lan93}. When 
$\mu = \leb^1$, where $ \leb^1$ is the one dimensional Lebesgue measure, we write 
$\L^p(I; \H) := \L^p(I,\mu; \H)$.

We recall that a \emph{$\H$-valued measure on $I$} is a map $\nu : \borel(I) \function \H$ such that 
$\nu(\bigcup_{n=1}^{\infty} B_{n})$ $=$ $\sum_{n = 1}^{\infty} \nu(B_{n})$ for every sequence $(B_{n})$ of mutually disjoint sets in $\borel(I)$. The \emph{total variation of $\nu$} is the positive measure 
$\vartot{\nu} : \borel(I) \function \clint{0,\infty}$ defined by
\begin{align}\label{tot var measure}
  \vartot{\nu}(B)
  := \sup\left\{\sum_{n = 1}^{\infty} \norm{\nu(B_{n})}{}\ :\ 
                 B = \bigcup_{n=1}^{\infty} B_{n},\ B_{n} \in \borel(I),\ 
                 B_{h} \cap B_{k} = \varnothing \text{ if } h \neq k\right\}. \notag
\end{align}
The vector measure $\nu$ is said to be \emph{with bounded variation} if $\vartot{\nu}(I) < \infty$. In this case the equality $\norm{\nu}{} := \vartot{\nu}(I)$ defines a complete norm on the space of measures with bounded variation (see, e.g. \cite[Chapter I, Section  3]{Din67}). 

If $\mu : \borel(I) \function \clint{0,\infty}$ is a positive bounded Borel measure and if $g \in \L^1(I,\mu;\H)$, then $g\mu : \borel(I) \function \H$ denotes the vector measure defined by 
\begin{equation}\label{gmu}
  g\mu(B) := \int_B g\de \mu, \qquad B \in \borel(I). \notag
\end{equation} 

Assume that $\nu : \borel(I) \function \H$ is a vector measure with bounded variation and $f : I \function \H$ and 
$\phi : I \function \mathbb{R}$ are two \emph{step maps with respect to $\nu$}, i.e. there exist 
$f_{1}, \ldots, f_{m} \in \H$, $\phi_{1}, \ldots, \phi_{m} \in \H$ and $A_{1}, \ldots, A_{m} \in \borel(I)$ mutually disjoint such that $\vartot{\nu}(A_{j}) < \infty$ for every $j$ and $f = \sum_{j=1}^{m} \indicator_{A_{j}} f_{j}$, 
$\phi = \sum_{j=1}^{m} \indicator_{A_{j}} \phi_{j}.$ Here $\indicator_{S} $ is the characteristic function of a set $S$, i.e. $\indicator_{S}(x) := 1$ if $x \in S$ and $\indicator_{S}(x) := 0$ if $x \not\in S$. For such step maps we define 
$\int_{I} \duality{f}{\de\nu} := \sum_{j=1}^{m} \duality{f_{j}}{\nu(A_{j})} \in \mathbb{R}$ and
$\int_{I} \phi \de \nu := \sum_{j=1}^{m} \phi_{j} \nu(A_{j}) \in \H$.

If $\Step(\vartot{\nu};\H)$ (resp. $\Step(\vartot{\nu})$) is the set of $\H$-valued (resp. real valued) step maps with respect to $\nu$, then the maps
$\Step(\vartot{\nu};\H)$ $\function$ $\H : f \longmapsto \int_{I} \duality{f}{\de\nu}$ and
$\Step(\vartot{\nu})$ $\function$ $\H : \phi \longmapsto \int_{I} \phi \de \nu$ 
are linear and continuous when $\Step(\vartot{\nu};\H)$ and $\Step(\vartot{\nu})$ are endowed with the 
$\L^{1}$-seminorms $\norm{f}{\L^{1}(\vartot{\nu};\H)} := \int_I \norm{f}{} \de \vartot{\nu}$ and
$\norm{\phi}{\L^{1}(\vartot{\nu})} := \int_I |\phi| \de \vartot{\nu}$. Therefore they admit unique continuous extensions 
$\mathsf{I}_{\nu} : \L^{1}(\vartot{\nu};\H) \function \mathbb{R}$ and 
$\mathsf{J}_{\nu} : \L^{1}(\vartot{\nu}) \function \H$,
and we set 
\[
  \int_{I} \duality{f}{\de \nu} := \mathsf{I}_{\nu}(f), \quad
  \int_{I} \phi\, \de\nu := \mathsf{J}_{\nu}(\phi),
  \qquad f \in \L^{1}(\vartot{\nu};\H),\quad \phi \in \L^{1}(\vartot{\nu}).
\]

The following results (cf., e.g., \cite[Section III.17.2-3, p. 358-362]{Din67}) provide the connection between functions with bounded variation and vector measures which will be implicitly used in this paper.

\begin{Thm}\label{existence of Stietjes measure}
For every $f \in \BV(I;\H)$ there exists a unique vector measure of bounded variation 
$\nu_{f} : \borel(I) \function \H$ such that 
\begin{align}
  \nu_{f}(\opint{c,d}) = f(d-) - f(c+), \qquad \nu_{f}(\clint{c,d}) = f(d+) - f(c-), \notag \\ 
  \nu_{f}(\clsxint{c,d}) = f(d-) - f(c-), \qquad \nu_{f}(\cldxint{c,d}) = f(d+) - f(c+). \notag 
\end{align}
whenever $\inf I \le c < d \le \sup I$ and the left hand side of each equality makes sense.

\noindent Conversely, if $\nu : \borel(I) \function \H$ is a vector measure with bounded variation, and if 
$f_{\nu} : I \function \H$ is defined by $f_{\nu}(t) := \nu(\clsxint{\inf I,t} \cap I)$, then $f_{\nu} \in \BV(I;\H)$ and 
$\nu_{f_{\nu}} = \nu$.
\end{Thm}

\begin{Prop}
Let $f  \in \BV(I;\H)$, let $g : I \function \H$ be defined by $g(t) := f(t-)$, for $t \in \Int(I)$, and by $g(t) := f(t)$, if 
$t \in \partial I$, and let $V_{g} : I \function \mathbb{R}$ be defined by $V_{g}(t) := \pV(g, \clint{\inf I,t} \cap I)$. Then  
$\nu_{g} = \nu_{f}$ and $\vartot{\nu_{f}}(I) = \nu_{V_{g}}(I) = \pV(g,I)$.
\end{Prop}

The measure $\nu_{f}$ is called the \emph{Lebesgue-Stieltjes measure} or \emph{differential measure} of $f$. Let us see the connection between  the differential measure and the distributional derivative. If $f \in \BV(I;\H)$ and if 
$\overline{f}: \mathbb{R} \function \H$ is defined by
\begin{equation}\label{extension to R}
  \overline{f}(t) :=
  \begin{cases}
    f(t) 	& \text{if $t \in I$}, \\
    f(\inf I)	& \text{if $\inf I \in \mathbb{R}$, $t \not\in I$, $t \le \inf I$},\\
    f(\sup I)	& \text{if $\sup I \in \mathbb{R}$, $t \not\in I$, $t \ge \sup I,$}
  \end{cases}
\end{equation}
then, as in the scalar case, it turns out (cf. \cite[Section 2]{Rec11a}) that $\nu_{f}(B) = \D \overline{f}(B)$ for every 
$B \in \borel(\mathbb{R})$, where $\D\overline{f}$ is the distributional derivative of $\overline{f}$, i.e.
\[
  - \int_\mathbb{R} \varphi'(t) \overline{f}(t) \de t = \int_{\mathbb{R}} \varphi \de \D \overline{f}, 
  \qquad \forall \varphi \in \Czero_{c}^{1}(\mathbb{R};\mathbb{R}),
\]
where $\Czero_{c}^{1}(\mathbb{R};\mathbb{R})$ is the space of continuously differentiable functions on 
$\mathbb{R}$ with compact support. Observe that $\D \overline{f}$ is concentrated on $I$: 
$\D \overline{f}(B) = \nu_f(B \cap I)$ for every $B \in \borel(I)$, hence in the remainder of the paper, if 
$f \in \BV(I,\H)$ then we will simply write
\begin{equation}
  \D f := \D\overline{f} = \nu_f, \qquad f \in \BV(I;\H),
\end{equation}
and from the previous discussion it follows that 
\begin{equation}\label{D-TV-pV}
  \norm{\D f}{} = \vartot{\D f}(I) = \norm{\nu_f}{}  = \pV(f,I), \qquad \forall f \in \BV(I;\H).
\end{equation}
If $I$ is bounded and $p \in \clint{1,\infty}$, then the classical Sobolev space $\Sob^{1,p}(I;\H)$ consists of those functions $f \in \Czero(I;\H)$ for which $\D f = g\leb^1$ for some $g \in \L^p(I;\H)$ and 
$\Sob^{1,\infty}(I;\H) = \Lip(I;\H)$. Let us also recall that if $f \in \Sob^{1,1}(I;\H)$ then the derivative $f'(t)$ exists 
$\leb^1$-a.e. in $t \in I$, $\D f = f' \leb^1$, and $\V(f,I) = \int_I\norm{f'(t)}{}\de t$ (see e.g. \cite[Appendix]{Bre73}).

In \cite[Lemma 6.4 and Theorem 6.1]{Rec11a} it is proved that

\begin{Prop}\label{P:BV chain rule}
Assume that $I, J \subseteq \ar$ are intervals and that $h : I \function J$ is nondecreasing.
\begin{itemize}
\item[(i)]
  $\D h(h^{-1}(B)) = \leb^{1}(B)$ for every $B \in \borel(h(\cont(h)))$.
\item[(ii)]
 If $f \in \Lip(J;\H)$ and $g : I \function \H$ is defined by
\begin{equation}\notag
  g(t) := 
  \begin{cases}
    f'(h(t)) & \text{if $t \in \cont(h)$} \\
    \ \\
    \dfrac{f(h(t+)) - f(h(t-))}{h(t+) - h(t-)} & \text{if $t \in \discont(h)$}
  \end{cases},
\end{equation}
then $f \circ h \in \BV(I;\H)$ and $\D\ \!(f \circ h) = g \D h$. This result holds even if $f'$ is replaced by any of its 
$\leb^{1}$-representatives.
\end{itemize}
\end{Prop}

\subsection{Set-valued functions with bounded retraction}

Let us now recall the asymmetric extendend distance called \emph{excess}. The material in this subsection 
is taken from \cite{Mor74a, Mor74b}.

\begin{Def}
If $\parti(\H)$ denotes the family of all subsets of $\H$, then the \emph{excess} is the function 
$\e : \parti(\H)\times\parti(\H) \function \clint{0,\infty}$ defined by
\begin{equation}\notag
  \e(\Z_1,\Z_2) := \sup_{z_1 \in \Z_1} \d(z_1,\Z_2) = 
  \inf\{\rho \ge 0\ :\ \Z_1 \subseteq \Z_2 + \overline{B}_\rho(0)\}, \qquad \Z_1, \Z_2 \in \parti(\H).
\end{equation}
The extended number $\e(\Z_1,\Z_2)$ is also called \emph{excess of $\Z_1$ over $\Z_2$}.
\end{Def}

Observe that $\e(\void, \Z) = 0$ for every $\Z \in \parti(\H)$ and $\e(\Z, \void) = \infty$ for every 
$\Z \in \parti(\H) \setmeno \{\void\}$.
The following facts are easy to prove (see \cite[Section 2a]{Mor74b}):
\begin{alignat}{3}
  & \e(\Z_1,\Z_2) = 0 \ \Longleftrightarrow \Z_1 \subseteq \overline{\Z_2} & \qquad & \forall \Z_1, \Z_2 \in \parti(\H), \notag \\
  & \e(\Z_1,\Z_2) \le \e(\Z_1,\Z_3) + \e(\Z_3,\Z_2) & \qquad & \forall \Z_1, \Z_2, \Z_3 \in \parti(\H), \notag \\
  & \e(\Z_1,\Z_2) = \e(\overline{\Z_1},\Z_2) = \e(\Z_1,\overline{\Z_2}) = \e(\overline{\Z_1},\overline{\Z_2})
  & \qquad & \forall \Z_1, \Z_2 \in \parti(\H). \notag
\end{alignat}

\begin{Def}
Given an interval $I \subseteq \ar$, $\mathscr{F} \subseteq \parti(\H)$, a 
function $\C : I \function \mathscr{F}$, and a subinterval $J \subseteq I$, the 
\emph{retraction of $\C$ on $J$} is defined by
\begin{equation}\label{retraction}
  \ret(\C,J) := 
  \sup\left\{
           \sum_{j=1}^{m} \e(\C(t_{j-1}),\C(t_{j}))\ :\ m \in \en,\ t_{j} \in J\ \forall j,\ t_{0} < \cdots < t_{m} 
         \right\}.
\end{equation}
If $\ret(\C,I) < \infty$ we say that \emph{$\C$ is of bounded retraction on $I$} and we set 
\[
  \BR(I;\mathscr{F}) := \big\{\C : I \function \mathscr{F}\ :\ \ret(\C,I) < \infty\big\}, 
  \qquad \mathscr{F} \subseteq \parti(\H).
\] 
If  $\ret(\C,J) < \infty$ for every interval $J$ compact in $I$, we say that \emph{$\C$} has 
\emph{local bounded retraction} (or \emph{finite retraction on $I$}) and we set 
\[
\BR_\loc(I;\mathscr{F}) := \big\{\C : I \function \mathscr{F}\ :\ \ret(\C,J) < \infty\ \text{$\forall J$ compact in $I$}\big\}, \qquad  \mathscr{F} \subseteq \parti(\H).
\]
\end{Def}

\begin{Def}
Assume that $I$ is an interval of $\ar$ and that $\C \in \BR_\loc(I;\parti(\H))$. If $t \in I$ we say that 
\emph{$\C$ is left-continuous at $t$} if there exists $\lim_{s \to t-} \e(\C(s), \C(t)) = 0$, and we say that 
\emph{$\C$ is right-continuous at $t$} if there exists $\lim_{s \to t+} \e(\C(t), \C(s)) = 0$, with the convention that 
$\lim_{s \to t-} \e(\C(s), \C(t)) = 0$ if $t = \inf I$, and that $\lim_{s \to t+} \e(\C(t), \C(s)) = 0$ if $t = \sup I$. We say that \emph{$\C$ is continuous at $t$} if it is left-continuous and right-continuous at $t$. We set
\[
  \cont(\C) := \{t \in I\ :\ \text{\emph{$\C$ is continuous at $t$}}\}, \qquad \discont(\C) := I \setmeno \cont(\C).
\]
and
\begin{equation}
  \C(t+) 
  := \{x \in \H\ :\ \lim_{s \to t+} \d(x,\C(s)) = 0\}, \qquad t \in I \cup \inf I,
\end{equation}
\begin{equation}
  \C(t-) 
  := \{x \in \H\ :\ \lim_{s \to t-} \d(x,\C(s)) = 0\}, \qquad t \in I \cup \sup I,
\end{equation}
together with $\C(t-) := \C(t)$ if $t = \inf I$, and $\C(t+) := \C(t)$ if $t = \sup I$.
\end{Def}

The following proposition is proved in \cite[Proposition 4b, Proposition 4c, Section 4e]{Mor74b}.

\begin{Prop}
Assume that $I$ is an interval of $\ar$ such that $t_0 := \inf I \in I$, and assume that $\C \in \BR_\loc(I;\parti(\H))$. 
Let $\ell_\C : I \function \clsxint{0,\infty}$ be defined by
\begin{equation}\label{ret}
  \ell_\C(t) := \ret(\C;\clint{t_0,t}), \qquad t \in I.
\end{equation}
Then $\ret(\C;\clint{a,b}) = \ret(\C;\clint{a,c}) + \ret(\C;\clint{c,b})$ whenever $a, b \in I$ with $a < c < b$, and
\[
  \e(\C(s),\C(t)) \le \ell_\C(t) - \ell_\C(s) \qquad \forall t, s \in I,\ s \le t,
\]
thus if $\C(t_*) \neq \void$ for some $t_* \in I$, then $\C(t) \neq \void$ for every $t \ge t_*$. \newline
If $t \in I$, $t \neq \sup I$, then $\C(t+)$ is a closed subset of $\H$ and
\[
  \ell_\C(t+) - \ell_\C(t) = \lim_{s \to t+} \e(\C(t), \C(s)) \le \e(\C(t), \C(t+))
\] 
and the following four statements \emph{(a)-(d)} are equivalent:
\begin{itemize}
\item[(a)]
$\C$ is right-continuous at $t$,
\item[(b)]
$\ell_\C$ is right-continuous at $t$,
\item[(c)]
$\C(t) \subseteq \C(t+)$,
\item[(d)]
$\e(\C(t),\C(t+)) = 0$.
\end{itemize} 
Instead if $t \in I$, $t \neq \inf I$, then $\C(t-)$ is a closed subset of $\H$ and
\[
   \e(\C(t-), \C(t)) \le \ell_\C(t) - \ell_\C(t-) = \lim_{s \to t-} \e(\C(s), \C(t)), 
\] 
\[
   \e(\C(s), \C(t-)) \le \ell_\C(t-) - \ell_\C(s) \qquad \forall s, t \in I,\ s < t,
\] 
thus we have that $\C(t-) \neq \void$ if $\C(s) \neq \void$ for some $s < t$.
Therefore the following four statements \emph{(e)-(h)} are equivalent:
\begin{itemize}
\item[(e)]
$\C$ is left-continuous at $t$,
\item[(f)]
$\ell_\C$ is left-continuous at $t$,
\item[(g)]
$\C(t-) \subseteq \overline{\C(t)}$,
\item[(h)]
$\e(\C(t-),\C(t)) = 0$.
\end{itemize} 
\end{Prop}

Thus we can give the following definitions.

\begin{Def}
Assume that $I$ is an interval of $\ar$ such that $t_0 := \inf I \in I$, and that $r \in \clsxint{0,\infty}$. If 
$\C \in \BR_\loc(I;\Conv_r(\H))$ we set 
\[
\cont(\C) := \{t \in I\ :\ \e(\C(t),\C(t+)) = \e(\C(t-),\C(t)) = 0\}
\]
and $\discont(\C) := I \setmeno \cont(\C)$.
Moreover we set 
\[
\BR^\r(I;\Conv_r(\H)) := \{\C \in \BR(I;\Conv_r(\H))\ :\ \e(\C(t),\C(t+)) = 0\ \forall t \in I\},
\]
\[
\BR^\r_\loc(I;\Conv_r(\H)) := \{\C \in \BR_\loc(I;\Conv_r(\H))\ :\ \e(\C(t),\C(t+)) = 0\ \forall t \in I\},
\] 
\[
  \Czero\BR_\loc(I;\Conv_r(\H)) := \{\C \in \BR^r_\loc(I;\Conv_r(\H))\ :\ \cont(\C) = I\}.
\].
\end{Def}

  
\section{The $1$-Lipschitz continuous case}\label{S:Lip case}

Within this section we solve the prox-regular sweeping process driven by a moving set which is $1$-Lipschitz continuous with respect to the excess. 
 
\begin{Thm}\label{Lip theorem}
Assume that $r \in \opint{0,\infty}$ and let $\Conv_r(\H)$ be the family of $r$-prox-regular subsets of $\H$. Assume that $\C : \clsxint{0,\infty} \function \Conv_r(\H)$ is $1$-Lipschitz continuous with respect to the excess $\e$ in the following sense:
\begin{equation}\label{C 1-e-Lip}
  \e(\C(s),\C(t)) \le t-s \qquad \forall t, s \in \clsxint{0,\infty}, s < t.
\end{equation}
If $y_0 \in \C(0) + B_r(0)$ then there exists a unique function $y : \clsxint{0,\infty} \function \H$ such that 
\begin{alignat}{3}
  & y(t) \in \C(t) & \quad & \forall t \in \clsxint{0,\infty}, \label{yinC-1Lip} \\
  & y'(t) \in -N_{\C(t)}(y(t)) & \quad & \text{for $\leb^{1}$-a.e. $t \in \clsxint{0,\infty}$}, \label{diffincl-1Lip} \\
  & y(0) = \P_{\C(0)}(y_{0}). &  \label{incond-1Lip}  
\end{alignat}
Moreover $y \in \Lip(\clsxint{0,\infty};\H)$ and $\Lipcost[y] \le 1$. 
We will denote this unique solution by 
$\Sw(\C,y_0)$, so that if 
\[
  \textsl{D}_{\Lip} := 
  \{(\C,y_0) \in \BR_\loc(\clsxint{0,\infty};\Conv_r(\H)) \times \H\ :\  
  \text{\eqref{C 1-e-Lip} holds},\ y_0 \in \C(0) + B_r(0)\},  
\]
then we can define $\Sw$ $:$ $\textsl{D}_{\Lip} \function \Lip(\clsxint{0,\infty};\H)$, 
the solution operator associating with every pair $(\C,y_0)$ the only $y$ satisfying \eqref{yinC-1Lip}--\eqref{incond-1Lip}.
\end{Thm}

\begin{proof}
First of all observe that the uniqueness of solutions can be inferred arguing exactly as in 
\cite[Proposition 3.6]{Thi16}, where only the hypomonotonicity is used. The argument in our situation is actually simpler bacause the functions involved here are absolutely continuous and the standard Gronwall lemma can be exploited.

Concerning existence, let $n_r \in \en$ be such that
\begin{equation}\label{LT1}
  \frac{1}{2^n} < r \qquad \forall n \in \en,\ n \ge n_r,
\end{equation}
and for every $n \in \en$, $n \ge n_r$, let $(t_j^n)_{j \in \en}$ be the sequence defined by
\begin{equation}\label{LT2}
  t_j^n := \frac{j}{2^n}, \qquad j \in \en.
\end{equation}
Since $y_0 \in \C(0) + B_r(0)$, from the $r$-prox-regularity of $\C(0)$ it follows that there exists a unique projection $\P_{\C(0)}(y_0) \in \C(0)$. So now we fix $n \ge n_r$ and we prove by induction that there is a sequence 
$(y_j^n)_{j=1}^\infty$ such that:
\begin{alignat}{3}
   & y_0^n := \P_{\C(0)}(y_0),   
      \label{implicit scheme for Lip-sweeping processes1} \\
   & y_j^n \in \C(t_j^n) & \qquad & \forall j \in\ \en,  
      \label{implicit scheme for Lip-sweeping processes2} \\
   & -\dfrac{y_j^n -y_{j-1}^n}{t_j^n - t_{j-1}^n} \in N_{\C(t_j^n)}(y_j^n) & \qquad & \forall j \in\ \en, 
      \label{implicit scheme for Lip-sweeping processes3}
\end{alignat}
which is an implicit time discretization scheme for \eqref{yinC-1Lip}--\eqref{incond-1Lip} usually called 
``catching-up algorithm". As $y_0^n \in \C(0)$, by virtue of \eqref{C 1-e-Lip} and \eqref{LT1} we have that 
$\d(y_0^n, \C(t_1^n)) \le \e\left(\C(0),\C(t_1^n)\right) \le t_1^n = 1/2^n < r$, therefore there exists 
$y_1^n := \P_{\C(t_1^n)}(y_0^n)$, which implies that $y_0^n-y_1^n \in N_{\C(t_1^n)}(y_1^n)$, or equivalently that 
\eqref{implicit scheme for Lip-sweeping processes3} holds with $j = 1$. Now for any integer $j > 1$ let us assume that $y_{j-1}^n$ satisfies the inclusions $y_{j-1}^n \in \C(t_{j-1}^n)$ and
$-(y_{j-1}^n -y_{j-2}^n)/(t_{j-1}^n - t_{j-2}^n) \in N_{\C(t_{j-1}^n)}(y_{j-1}^n)$. Therefore we have 
\[
  \d(y_{j-1}^n, \C(t_j^n)) \le \e(\C(t_{j-1}^n),\C(t_j^n)) \le t_j^n - t_{j-1}^n = \frac{1}{2^n} < r,
\]
thus there exists $y_j^n := \P_{\C(t_j^n)}(y_{j-1}^n)$, and
$y_{j-1}^n-y_j^n  \in N_{\C(t_j^n)}(y_j^n)$. Thus we have proved that there is a sequence $(y_j^n)_{j=1}^\infty$ such that
\eqref{implicit scheme for Lip-sweeping processes1}--\eqref{implicit scheme for Lip-sweeping processes3} hold, or equivalently such that \eqref{implicit scheme for Lip-sweeping processes1} holds together with
\begin{equation}\label{LT3}
    y_j^n = \P_{\C(t_j^n)}(y_{j-1}^n) \qquad \forall j \in \en.
\end{equation}
Now for every $n \ge n_r$ we define the step function $y_n : \clsxint{0,\infty} \function \H$ by 
\begin{align}
  & y_n(0) := y_0^n, \\
  & y_n(t) := y_{j}^n \qquad \text{if $t \in \cldxint{t_{j-1}^{n}, t_{j}^{n}}$}, \quad j \in \en, 
\end{align}
and let $x_n : \clsxint{0,\infty} \function \H$ be the piecewise affine interpolant of $(y_j^n)_{j}$, i.e.
\begin{equation}
  x_n(t) := y_{j-1}^n + \frac{t-t^n_{j-1}}{t_j^n - t_{j-1}^n}(y_j^n - y_{j-1}^n)
  \qquad \text{if $t \in \clsxint{t_{j-1}^n, t_j^n}$}, \quad j \in \en.
\end{equation}
We have by \eqref{implicit scheme for Lip-sweeping processes3} that 
\begin{equation}\label{x_n' in N}
  -x_n'(t) = \frac{y_{j-1}^n - y_j^n}{t_j^n - t_{j-1}^n} \in N_{C(t_j^n)}(y_j^n) = N_{C(t_j^n)}(y(t))
  \qquad \forall t \in \opint{t_{j-1}^n, t_j^n}, \ \forall j \in \en,
\end{equation}
and by \eqref{LT3} that
\begin{align}
  \norm{x_n'(t)}{} 
  & =   \frac{\norm{y_j^n - y_{j-1}^n}{}}{t_j^n - t_{j-1}^n} 
  =   \frac{\norm{\P_{\C(t_j^n)}(y_{j-1}^n) - y_{j-1}^n}{}}{t_j^n - t_{j-1}^n} \notag \\
  & =   \frac{\d(y_{j-1}^n, \C(t_j^n))}{t_j^n - t_{j-1}^n}
  \le \frac{\e(\C(t_{j-1}^n), \C(t_j^n))}{t_j^n - t_{j-1}^n} \qquad \forall t \in \opint{t_{j-1}^n, t_j^n},
\end{align}
therefore \eqref{C 1-e-Lip} yields
\begin{equation}\label{xn 1-Lip}
  \norm{x_n'(t)}{}  \le 1 \qquad \forall t \neq t_j^n,
\end{equation}
and $x_n$ is $1$-Lipschitz continuous on $\clsxint{0,\infty}$.
Now let $\theta_n : \clsxint{0,\infty} \function \clsxint{0,\infty}$ be defined by 
\begin{equation}\label{tetan}
  \theta_n(0) = 0, \qquad \theta_n(t) := t_{j}^n \quad \text{if $t \in \cldxint{t_{j-1}^n,t_j^n}$}, \ j \in \en. \notag
\end{equation}
so that
\begin{equation}\label{LT4}
  0 \le \theta_n(t) - t \le \frac{1}{2^n},
\end{equation}
\begin{equation}
  x_n(\theta_n(t)) = x_n(t_j^n) = y_j^n \in \C(t_j^n) \qquad \forall t \in \cldxint{t_{j-1}^n,t_j^n},\ \forall j \in \en,
\end{equation}
and, by \eqref{xn 1-Lip} and \eqref{LT4},
\begin{equation}\label{LT-xn-x}
  \norm{x_{n}(\theta_{n}(t)) - x_{n}(t)}{} \le |\theta_n(t) - t| \le 1/2^n.
\end{equation}
Now let us fix $n \ge n_r$ 
and $t \ge 0$ such that $t \neq t_j^n$ and $t \neq t_{j}^{n+1}$ for all $j \in \en \cup \{0\}$. We have that
\begin{align}
  \lduality{x'_{n+1}(t) - x'_n(t)}{x_{n+1}(t) - x_n(t)} 
  & = \lduality{x'_{n+1}(t) - x'_n(t)}{x_{n+1}(t) - x_{n+1}(\theta_{n+1}(t))} \notag \\
     &  \phantom{=\ }  +  \lduality{x'_{n+1}(t) - x'_n(t)}{x_{n+1}(\theta_{n+1}(t)) - x_{n}(\theta_{n}(t))} \notag \\
     &  \phantom{=\ }  +  \lduality{x'_{n+1}(t) - x'_n(t)}{x_{n}(\theta_{n}(t)) - x_n(t)} \notag \\
 & \le \norm{x'_{n+1}(t) - x'_n(t)}{}\norm{x_{n+1}(t) - x_{n+1}(\theta_{n+1}(t))}{} \notag \\
     &  \phantom{=\ }  +  \lduality{x'_{n+1}(t) - x'_n(t)}{x_{n+1}(\theta_{n+1}(t)) - x_{n}(\theta_{n}(t))} \notag \\
     &  \phantom{=\ }  +  \norm{x'_{n+1}(t) - x'_n(t)}{}\norm{x_{n}(\theta_{n}(t)) - x_n(t)}{}, \notag
\end{align}
therefore, since $\norm{x'_{n+1}(t) - x'_n(t)}{} \le 2$ by \eqref{xn 1-Lip}, 
using \eqref{LT-xn-x} 
we infer that 
\begin{align}\label{d/dt <}
 \frac{\de}{\de t}\norm{x_{n+1}(t) - x_n(t)}{}^2 
 & = 2\duality{x'_{n+1}(t) - x'_n(t)}{x_{n+1}(t) - x_n(t)} \notag \\
 & \le 2\lduality{x'_{n+1}(t) - x'_n(t)}{x_{n+1}(\theta_{n+1}(t)) - x_{n}(\theta_{n}(t))} + \frac{6}{2^{n}}.
\end{align}
In order to estimate properly the right-hand side of \eqref{d/dt <}, we
first observe that from \eqref{LT-xn-x} we obtain 
\begin{align}
  & \norm{x_{n+1}(\theta_{n+1}(t)) - x_{n}(\theta_{n}(t))}{}^2 \notag \\
  & \le  4\norm{x_{n+1}(\theta_{n+1}(t)) - x_{n+1}(t)}{}^2 +
   4\norm{x_{n+1}(t) - x_{n}(t)}{}^2 +
    2\norm{x_{n}(t) - x_{n}(\theta_{n}(t))}{}^2 \notag \\
  & \le  \frac{4}{2^{2n+2}} +
   4\norm{x_{n+1}(t) - x_{n}(t)}{}^2 +
    \frac{2}{2^{2n}} = \frac{4}{2^{2n}} + 4\norm{x_{n+1}(t) - x_{n}(t)}{}^2. \label{stima freq} 
\end{align}
Then we take $j \in \en$ such that $t \in \opint{t_{j-1}^n, t_j^n}$ and observe that
\[
  t_{j-1}^n = t_{2j-2}^{n+1} <  t_{2j-1}^{n+1} < t_{2j}^{n+1} = t_j^n,
\]
thus $\theta_{n+1}(t) \in \{t_{2j-1}^{n+1}, t_{2j}^{n+1}\}$ and 
$x_n(t_{j-1}^n) = y_{j-1}^{n} \in \C(t_{j-1}^{n}) = \C(t_{2j-2}^{n+1})$, therefore thanks to \eqref{C 1-e-Lip} and 
\eqref{tetan},
\begin{align}
 \d(x_n(t_{j-1}^n),\C(\theta_{n+1}(t))) 
 & \le \e(\C(t_{2j-2}^{n+1}), \C(\theta_{n+1}(t))) \notag \\
 & \le \theta_{n+1}(t) - t_{2j-2}^{n+1} \le t_{2j}^{n+1} - t_{2j-2}^{n+1} = 1/2^{n}. \label{LT6}
\end{align}
Therefore by \eqref{LT1} there exists $\P_{\C(\theta_{n+1}(t))}(x_n(t_{j-1}^n))$, and  
as by \eqref{x_n' in N} we have $-x'_{n+1}(t)$ $\in$ $N_{\C(\theta_{n+1}(t))}(x_{n+1}(\theta_{n+1}(t)))$, using 
the $r$-prox-regularity of $\C(\theta_{n+1}(t))$, we can write
\begin{align}
  & \lduality{x'_{n+1}(t)}{x_{n+1}(\theta_{n+1}(t)) - x_{n}(\theta_{n}(t))} \notag \\
  & =  \lduality{x'_{n+1}(t)}{x_{n+1}(\theta_{n+1}(t)) - \P_{\C(\theta_{n+1}(t))}(x_n(t_{j-1}^n))}  \notag \\
  & \phantom{=\ } + 
          \lduality{x'_{n+1}(t)}{\P_{\C(\theta_{n+1}(t))}(x_n(t_{j-1}^n)) - x_{n}(\theta_{n}(t))} \notag \\
  & \le \frac{\lnorm{x_{n+1}(\theta_{n+1}(t)) - \P_{\C(\theta_{n+1}(t))}(x_n(t_{j-1}^n))}{}^2}{2r} +
           \lnorm{\P_{\C(\theta_{n+1}(t))}(x_n(t_{j-1}^n))) - x_{n}(\theta_{n}(t))}{}. \label{LT5}
\end{align}
Now observe that from \eqref{xn 1-Lip}, \eqref{tetan}, \eqref{LT6}, and \eqref{stima freq}, we infer that
\begin{align}
  & \lnorm{x_{n+1}(\theta_{n+1}(t)) - \P_{\C(\theta_{n+1}(t))}(x_n(t_{j-1}^n))}{}^2 \notag \\
 & \le 2\lnorm{x_{n+1}(\theta_{n+1}(t)) - x_n(\theta_n(t))}{}^2 \notag \\
 & \phantom{\le \ } +         
  4\lnorm{x_n(\theta_n(t)) - x_n(t_{j-1}^n)}{}^2  +
           4\lnorm{x_n(t_{j-1}^n) - \P_{\C(\theta_{n+1}(t))}(x_n(t_{j-1}^n))}{}^2 \notag \\
   & \le 2\lnorm{x_{n+1}(\theta_{n+1}(t)) - x_n(\theta_n(t))}{}^2  +
          4|\theta_n(t) - t_{j-1}^n|^2  +
           4\d(x_n(t_{j-1}^n), \C(\theta_{n+1}(t)))^2\notag \\
   & \le \frac{8}{2^{2n}} + 8\norm{x_{n+1}(t) - x_{n}(t)}{}^2 +
          \frac{4}{2^{2n}}  +
           \frac{4}{2^{2n}}, \label{LT7}
\end{align}
and that, by \eqref{xn 1-Lip}, \eqref{tetan} and \eqref{LT6},
\begin{align}
&  \lnorm{\P_{\C(\theta_{n+1}(t))}(x_n(t_{j-1}^n))) - x_{n}(\theta_{n}(t))}{} \notag \\
 & \le \lnorm{\P_{\C(\theta_{n+1}(t))}(x_n(t_{j-1}^n))) - x_n(t_{j-1}^n)}{} +
                   \lnorm{x_n(t_{j-1}^n) - x_{n}(\theta_{n}(t))}{}  \notag \\
     & \le   \d(x_n(t_{j-1}^n), \C(\theta_{n+1}(t))) +
                   |t_{j-1}^n - \theta_{n}(t)|  \le 
                \frac{1}{2^{n}}  +  \frac{1}{2^n}, \label{LT8}
 \end{align}
therefore, from \eqref{LT5}, \eqref{LT7}, and \eqref{LT8} we obtain that there exists a constant $K_1 > 0$, independent of $n$, such that
\begin{equation}
  \lduality{x'_{n+1}(t)}{x_{n+1}(\theta_{n+1}(t)) - x_{n}(\theta_{n}(t))} \le
  \frac{4}{r}\lnorm{x_{n+1}(t) - x_n(t)}{}^2  +
          \frac{K_1}{2^{n}}. \label{LT9}
\end{equation}
Now observe that $\theta_n(t) = t_j^n = t_{2j}^{n+1}$ and 
$x_{n+1}(t_{2j-2}^{n+1}) = y_{2j-2}^{n+1} \in \C(t_{2j-2}^{n+1})$, thus, thanks to \eqref{C 1-e-Lip} and \eqref{tetan},
\begin{align}
 \d(x_{n+1}(t_{2j-2}^{n+1}),\C(\theta_{n}(t))) 
 & \le \e(\C(t_{2j-2}^{n+1}), \C(\theta_{n}(t))) \notag \\
 & \le \theta_{n}(t) - t_{2j-2}^{n+1} = t_{2j}^{n+1} - t_{2j-2}^{n+1} = 1/2^{n}. \label{LT10}
\end{align}
Therefore by \eqref{LT1} there exists 
$\P_{\C(\theta(t))}(x_{n+1}(t_{2j-2}^{n+1}))$ and as $-x'_{n}(t) \in N_{\C(\theta_{n}(t))}(x_{n}(\theta_{n}(t)))$ by 
 \eqref{x_n' in N}, and using the $r$-prox-regularity of $\C(\theta_{n}(t))$, we can write
\begin{align}
  & \lduality{- x'_n(t)}{x_{n+1}(\theta_{n+1}(t)) - x_{n}(\theta_{n}(t))} \notag \\
  & =  \lduality{- x'_n(t)}{x_{n+1}(\theta_{n+1}(t)) - \P_{\C(\theta_n(t))}(x_{n+1}(t_{2j-2}^{n+1}))} \notag \\
   & \phantom{=\ } + 
         \lduality{- x'_n(t)}{\P_{\C(\theta_n(t))}(x_{n+1}(t_{2j-2}^{n+1})) - x_{n}(\theta_{n}(t))} \notag \\
  & \le \lnorm{x_{n+1}(\theta_{n+1}(t)) - \P_{\C(\theta_n(t))}(x_{n+1}(t_{2j-2}^{n+1}))}{} + 
           \frac{\lnorm{\P_{\C(\theta_n(t))}(x_{n+1}(t_{2j-2}^{n+1})) - x_{n}(\theta_{n}(t))}{}^2}{2r}. \label{LT11}
\end{align}
Now observe that, by \eqref{xn 1-Lip}, \eqref{tetan} and \eqref{LT10}, we have 
\begin{align}
&  \lnorm{x_{n+1}(\theta_{n+1}(t)) - \P_{\C(\theta_n(t))}(x_{n+1}(t_{2j-2}^{n+1}))}{} \notag \\
 & \le \lnorm{x_{n+1}(\theta_{n+1}(t)) - x_{n+1}(t_{2j-2}^{n+1})}{} +
                   \lnorm{x_{n+1}(t_{2j-2}^{n+1}) - \P_{\C(\theta_n(t))}(x_{n+1}(t_{2j-2}^{n+1}))}{}  \notag \\
 & \le    |\theta_{n+1}(t) - t_{2j-2}^{n+1}|  + \d(x_{n+1}(t_{2j-2}^{n+1}), \C(\theta_{n}(t))) 
                   \le 
                \frac{1}{2^{n}}  +  \frac{1}{2^n} \label{LT12}
 \end{align}
and that, from \eqref{xn 1-Lip}, \eqref{tetan}, \eqref{LT10}, and \eqref{stima freq}, we infer that
\begin{align}
  & \lnorm{\P_{\C(\theta_n(t))}(x_{n+1}(t_{2j-2}^{n+1})) - x_{n}(\theta_{n}(t))}{}^2 \notag \\
 & \le 4\lnorm{\P_{\C(\theta_n(t))}(x_{n+1}(t_{2j-2}^{n+1})) - x_{n+1}(t_{2j-2}^{n+1})}{}^2 
     \notag \\
 & \phantom{\le \ }  + 4\lnorm{x_{n+1}(\theta_{n+1}(t)) - x_{n+1}(t_{2j-2}^{n+1})}{}^2 +    
   2\lnorm{x_{n+1}(\theta_{n+1}(t)) - x_n(\theta_n(t))}{}^2
            \notag \\
   & \le 4\d(x_{n+1}(t_{2j-2}^{n+1}), \C(\theta_n(t)))^2 + 4|\theta_{n+1}(t) - t_{2j-2}^{n+1}|^2  +
   2\lnorm{x_{n+1}(\theta_{n+1}(t)) - x_n(\theta_n(t))}{}^2
          \notag \\
   & \le \frac{4}{2^{2n}} + \frac{4}{2^{2n}} + \frac{8}{2^{2n}}  + 8\norm{x_{n+1}(t) - x_{n}(t)}{}^2
           \label{LT13}
\end{align}
therefore, from \eqref{LT11}, \eqref{LT12}, and \eqref{LT13} we obtain that there exists a constant $K_2 > 0$, independent of $n$, such that
\begin{equation}
  \lduality{-x'_{n}(t)}{x_{n+1}(\theta_{n+1}(t)) - x_{n}(\theta_{n}(t))} \le
  \frac{4}{r}\lnorm{x_{n+1}(t) - x_n(t)}{}^2  +
          \frac{K_2}{2^{n}}. \label{LT14}
\end{equation}
Collecting together \eqref{d/dt <}, \eqref{LT9} and \eqref{LT14} we infer that there exists a constant $K_3 > 0$, independent of $n$ (depending only on $r$), such that
\begin{equation}
  \frac{\de}{\de t}\norm{x_{n+1}(t) - x_n(t)}{}^2 
 \le \frac{K_3}{2^n} + \frac{8}{r}\norm{x_{n+1}(t) - x_{n}(t)}{}^2. \label{LT15}
\end{equation}
By \eqref{LT15} and by the Gronwall lemma we infer that for every $T > 0$ there is a constant $K(T,r) > 0$ independent of $n$ (depending only on $r$ and $T$), such that
\[
\norm{x_{n+1}(t) - x_n(t)}{}^2 \le \frac{K(T,r)}{2^n} \qquad \forall t \in \clint{0,T}, \quad \forall n \ge n_r,
\]
so that
\[
  \sum_{n=n_r}^\infty \sup_{t \in \clint{0,T}}\norm{x_{n+1}(t) - x_n(t)}{} < \infty 
\]
and the sequence $x_n$ is uniformly Cauchy in $C(\clint{0,T};\H)$ for every $T > 0$. It follows that there exists a 
continuous function
$x : \clsxint{0,\infty} \function \H$ such that 
\begin{equation}
  x_n \to x \quad \text{uniformly on $\clint{0,T}, \qquad \forall T > 0$}.
\end{equation}
Since
\[
  \norm{y_n(t) - x_n(t)}{} =
   \bigg\|\frac{ y_j^n - y_{j-1}^n }{ t_j^n - t_{j-1}^n }\bigg\|  |t_j^n - t|
   \le  |t_j^n - t|
  \qquad \forall t \in \cldxint{t_{j-1}^n, t_j^n}
\]
we also deduce that 
\begin{equation}\label{yn>x}
  y_n \to x \quad \text{uniformly on $\clint{0,T}, \qquad \forall T > 0$}.
\end{equation}
If $t > 0$ and $t \in \cldxint{t_{j-1}^n, t_j^n}$, then we have
\begin{align}
  \d(y_n(t),\C(t)) & = \d(y_j^n, \C(t)) \notag\\
  &  \le \norm{y_j^n - y_{j-1}^n}{} + \d(y_{j-1}^n,\C(t)) \notag \\
  & \le   \d(y_{j-1}^n,\C(t_j^n)) + \e(\C(t_{j-1}^n),\C(t)) \\
  & \le \e(\C(t_{j-1}^n),\C(t_j^n)) + \e(\C(t_{j-1}^n),\C(t)) \le 2/2^n,
\end{align}
therefore by \eqref{yn>x} and by the closedness of $\C(t)$ we infer that 
\begin{equation}
  x(t) \in \C(t) \qquad \forall t \ge 0.
\end{equation}
Moreover from \eqref{xn 1-Lip} we also infer that $x \in \Lip(\clint{0,T};\H)$ and, at least for a subsequence which we do not relabel,
\begin{equation}
  x_n' \to x' \quad \text{weakly in $\L^2(\clint{0,T};\H)$} \qquad \forall T > 0.
\end{equation}
Now we can follow the argument of \cite[pp. 158-159]{ColThi10}.
By the Mazur's lemma there exists a sequence $v_n \in L^2(\clint{0,T};\H)$ such that
\begin{equation}
  v_n = \sum_{k = n}^{N(n)} \lambda_{j}^{(n)}x_j', \qquad \text{with $\sum_{k = n}^{N(n)} \lambda_{j}^{(n)} = 1$},
  \ \lambda_{j}^{(n)} \ge 0,
\end{equation}
and
\begin{equation}
  v_n \to x' \qquad \text{in $\L^2(\clint{0,T};\H)$}.
\end{equation}
At least for a further subsequence which we do not relabel we have that there exists $S \subseteq \clint{0,T}$
such that $\leb^1(\clint{0,T} \setmeno S) = 0$ and 
\begin{equation}
  v_n(t) \to x'(t) \quad \text{in $\H$,} \quad \forall t \in S.
\end{equation}
Therefore if $t \in S$ and if $z_t \in \C(t)$ are fixed arbitrarily, for every $k$ we have that
\begin{align}
   \duality{-x_k'(t)}{z_t - x(t)} 
   & = \duality{-x_k'(t)}{z_t - y_{k}(t)} + \duality{x_k'(t)}{y_k(t) - x(t)} \notag \\
   & \le \frac{\norm{x_k'(t)}{}}{2r}\norm{z_t - y_k(t)}{}^2 + \norm{y_k(t) - x(t)}{} \notag \\
   & \le \frac{1}{r}(\norm{z_t - x(t)}{}^2 + \norm{x(t)-y_k(t)}{}^2) + \norm{y_k(t) - x(t)}{}, \notag
\end{align}
hence if 
\[
  \eps_k := 
  \sup_{s \in \clint{0,T}} \bigg(\frac{1}{r}\norm{x(s)-y_k(s)}{}^2 + \norm{y_k(s) - x(s)}{}\bigg),
\]
we have that $\eps_k \to 0$ as $k \to \infty$ and 
\begin{equation}
\duality{-x_k'(t)}{z_t - x(t)} \le \frac{1}{r}\norm{z_t - x(t)}{}^2 + \eps_k \qquad \forall k \ge n_r.
\end{equation}
On the other hand 
\begin{align}
  \duality{-x'(t)}{z_t - x(t)} 
  & = \lim_{n \to \infty} \duality{-v_n(t)}{z_t - x(t)} \notag \\
  & = \lim_{n \to \infty} \sum_{k = n}^{N(n)} \lambda_{j}^{(n)}\duality{-x_k'(t)}{z_t - x(t)}, \notag 
\end{align}
hence we have
\begin{align}
  \duality{-x'(t)}{z_t - x(t)} 
  & \le \limsup_{n \to \infty} 
    \sum_{k = n}^{N(n)}\lambda_{j}^{(n)} \left(\frac{1}{r}\norm{z_t - x(t)}{}^2 + \eps_k\right) \notag \\
  & = \frac{1}{r}\norm{z_t - x(t)}{}^2 + \limsup_{n \to \infty} \sum_{k = n}^{N(n)}\lambda_{j}^{(n)} \eps_k \notag \\
  & \le \frac{1}{r}\norm{z_t - x(t)}{}^2 + \limsup_{n \to \infty} \sup_{k \ge n}\eps_k, \notag
\end{align}
but 
$\limsup_{n \to \infty} \sup_{k \ge n}\eps_k = 0$, therefore we have that
\[
 \duality{x'(t)}{z_t - x(t)} \le \frac{1}{r}\norm{z_t - x(t)}{}^2 \qquad \forall t \in S, \ \forall z_t \in \C(t) 
\]
and \eqref{diffincl-1Lip} is proved by virtue of Proposition \ref{propsigma}.
\end{proof}


\section{Prox-regular sweeping processes with bounded retraction}\label{proofs}

Let us start with a technical lemma. Let us remark that parts (ii)--(iv) of this lemma were proved in 
\cite[Lemmas 15.45, 15.47, 15.55]{Thi23b}, as explained in the proof, which we keep here for the sake of completeness. 

\begin{Lem}\label{L:lemma tecnico-2}
Assume that \eqref{H-prel} holds, $0 < \rho < r$, and that $\K$ is an $r$-prox-regular subset of $\H$. If 
$\K_\rho := \K + \overline{B}_\rho(0)$ then the following statements hold true.
\begin{itemize}
\item[(i)]
  $\K_\rho = \{x \in \H \ :\ \norm{x - \P_\K(x)}{} \le \rho\}$.
\item[(ii)]
 $\partial \K_\rho = \{x \in \H \ :\ \norm{x - \P_\K(x)}{} = \rho\}$.
\item[(iii)]
$\K_\rho$ is $(r-\rho)$-prox-regular.
\item[(iv)]
If $y \in \partial \K_\rho$ then $N_{\K_\rho}(y) = \left\{\lambda(y - \P_\K(y))\ :\ \lambda \ge 0\right\}$.
\end{itemize}
\end{Lem}

\begin{proof}
Concerning statement (i), if $x \in \K_\rho$ then $x = z + u$ for some $z \in \K$ and some 
$u \in \overline{B}_\rho(0)$, therefore 
$\norm{x - \P_\K(x)}{} = \d_\K(x) = \inf\{\norm{w-z-u}{}\ :\ w \in \K\} \le \norm{u}{} \le \rho$. Vice versa if $x \in \H$ and $\norm{x - \P_\K(x)}{} \le \rho$ then $x = \P_\K(x) + (x - \P_\K(x)) \in \K_\rho$.
Hence statement (ii) follows from the continuity of the mapping $x \longmapsto \norm{x - \P_\K(x)}{} = \d_\K(x)$ from $\K + B_{r}(0)$ into $\mathbb{R}$ (but it is also equivalent to the second statement in 
\cite[Lemma 15.55]{Thi23b}). Statement (iii) is proved in \cite[Lemma 15.45-(a)]{Thi23b}, for the convenience of the reader we report here the argument: thanks to \cite[Lemma 3.1]{BouThi02} we have that 
\begin{equation}
  \d_{\K_\rho}(y) = \d_\K(y) - \rho \qquad \forall  y \in \H \setmeno \K_\rho.
\end{equation} 
Therefore from the differentiability of $\d_\K$ on $\K_r \setmeno \K$ it follows that $\d_{\K_\rho}$ is differentiable in $(\K_\rho + B_{r-\rho}(0) )\setmeno \K_\rho$, and this implies statement (iii) by Thereom \ref{charact proxreg}.
We finally prove (iv). Let us observe that by \cite[Theorem 4.1]{PolRocThi00} we have that $\d_K$ is differentiable at $y$ and $\nabla\d_\K(y) = (y-\P_\K(y))/\norm{y-\P_\K(y)}{}$, therefore thanks to 
\cite[Therem 3.1-(2) and Theorem 3.4]{ClaSteWol95}
we infer that $N_{\K_\rho}(y) = \{\lambda\nabla\d_\K(y)\ :\ \lambda \ge 0\} = 
\{\lambda(y - \P_\K(y))\ :\ \lambda \ge 0\}$. 
Part (iv) can be also deduced from \cite[Lemma 15.47-(b)]{Thi23b}. 
\end{proof}

Now we define a particular multivalued curve which will be used to fill in the jumps of a discontinuous moving set 
with bounded retraction.

 \begin{Def}\label{D:F}
Assume that $\Z_0, \Z_1 \in \Conv_0(\H)$ and that $\rho := \e(\Z_0,\Z_1) < \infty$. 
We define the curve $\G_{(\Z_0,\Z_1)} : \clint{0,1} \function \Conv_0(\H)$ by 
\begin{equation}\label{catching-up geodesic-2}
  \G_{(\Z_0,\Z_1)}(t) := 
  \begin{cases}
    \Z_0 & \text{if $t=0$} \\
    \Z_1 + \overline{B}_{(1-t)\rho}(0)  & \text{if $0 < t \le 1$}
  \end{cases}
\end{equation}
\end{Def}

The curve defined in the previous definition is $1$-Lipschitz continuous with respect to the excess $\e$, indeed we can prove the following result.

\begin{Prop}
If $\Z_0, \Z_1 \in \Conv_0(\H)$ and $\G_{(\Z_0,\Z_1)} : \clint{0,1} \function \Conv_0(\H)$ is defined as in Definition \ref{D:F}, we have 
\begin{equation}\label{F geodesic}
  \e(\G(s),\G(t)) = (t-s)\e(\Z_0,\Z_1) \qquad \forall s, t \in \clint{0,1},\ s < t,
\end{equation}
and we can call $\G$ a \emph{$\e\ \!$-geodesic connecting $\Z_0$ to $\Z_1$}.
\end{Prop}

\begin{proof}
Let us set $D_\delta := \overline{B}_\delta(0)$. For every $\delta > 0$ we have 
$\G(0) = \Z_0 \subseteq \Z_1 + D_{\rho} = \Z_1 + D_{(1-t)\rho} + D_{t\rho}$ thus
$\e(\G(0),\G(t)) \le t\rho$. If $0 < s \le t$ we have
\begin{align}
  \G(s) 
    & = \Z_1 + D_{(1-s)\rho} \notag \\
    & = \Z_1 + D_{(1- t)\rho} + D_{(t-s)\rho} \notag \\   
    & = \G(t) + D_{(t-s)\rho}.
\end{align}
Therefore $\e(\G(s),\G(t)) \le (t - s)\rho = (t-s)\e(\Z_0,\Z_1)$. On the other hand we have
\begin{align}
\e(\Z_0,\Z_1) & \le \e(\Z_0,\G(s)) + \e(\G(s),\G(t)) + \e(\G(t),\Z_1) \notag \\
& \le s\e(\Z_0,\Z_1) + \e(\G(s),\G(t)) + (1-t)\e(\Z_0,\Z_1), \notag 
\end{align} hence
$(t-s)\e(\Z_0,\Z_1) \le \e(\G(s),\G(t))$ and \eqref{F geodesic} is proved.
\end{proof}

The next Lemma shows that the trajectory of the solution of the sweeping process driven by $\G_{(\Z_0,\Z_1)}$ is always a straight line segment connecting the initial datum $y_0$ to its projection to $\Z_1$.

\begin{Lem}\label{L:particular sweeping process}
Assume that $r \in \opint{0,\infty}$. Let $\Z_0 \in \Conv_{r}(\H)$ and $\Z_1 \in \Conv_r(\H)$ be such that 
$\rho := \e(\Z_0,\Z_1) < r$, and let 
$\G_{(\Z_0, \Z_1)} = \G : \clint{0,1} \function \Conv_r(\H)$ be defined by \eqref{catching-up geodesic-2}. If 
$y_0 \in \Z_0$, then let $t_0 \in \clint{0,1}$ be the unique number such that  
\begin{equation}\label{t_0 s.t. d(y_0,B) = (1-t_0)delta}
  \norm{y_0 - \P_{\Z_1}(y_0)}{} = (1-t_0)\rho
\end{equation}
($t_0$ is well defined because $\Z_1$ is $r$-prox-regular and $\e(\Z_1,\Z_2) < r$). If $y \in \Lip(\clint{0,1};\H)$ is defined by
\begin{equation}\label{solution for catching-up geodesic-2}
  y(t) :=
  \begin{cases}
    y_0 & \text{if $t \in \clsxint{0,t_0}$} \\
    y_0 + \dfrac{t-t_0}{1-t_0}(\P_{\Z_1}(y_0) - y_0) & \text{if $t_0 \neq 1$, $t \in \clsxint{t_0,1}$} \\
    \P_{\Z_1}(y_0) & \text{if $t = 1$}
  \end{cases},
\end{equation}
then
\begin{alignat}{3} 
  & y(t) \in \G_{(\Z_0, \Z_1)} (t) & \qquad & \forall t \in \clint{0,1}, 
     \label{y in G - Lip} \\
  & y'(t) \in -N_{\G_{(\Z_0, \Z_1)} (t)}(y(t)) & \qquad &  \text{for $\leb^{1}$-a.e. $t \in \clint{0,1}$}, 
      \label{diff. incl. G - Lip} \\
  & y(0) = y_0,
      \label{in. cond. G - Lip}
\end{alignat}
i.e. $y$ is the unique solution of the sweeping process driven by $\G_{(\Z_0,\Z_1)}$ with initial condition 
$y_0 \in \Z_0$. 
\end{Lem}

\begin{proof}
For simplicity let us write $\G := \G_{(\Z_0,\Z_1)}$.
If $y_0 \in \Z_1$ we have that $t_0 = 1$, $y(t) = y_0 = \P_{\Z_1}(y_0) \in \G(t)$ and $y'(t) = 0$ for every 
$t \in \clint{0,1}$, and we are done. Therefore we now assume that $y_0 \not\in \Z_1$, i.e. $t_0 < 1$, thus from 
Lemma \ref{L:lemma tecnico-2}-(ii) and formula \eqref{t_0 s.t. d(y_0,B) = (1-t_0)delta} we deduce that
\begin{equation}\label{y_0 in bordo B_(1-t)}
  y_0 \in \partial(\Z_1 + \overline{B}_{(1-t)\rho}) \iff t = t_0,
\end{equation}
(that is $t_0$ is the first time when $\partial(\Z_1 + \overline{B}_{(1-t)\rho})$ meets $y_0$). If $t \in \cldxint{t_0,1}$ we have that 
\[
  \norm{y(t) - \P_{\Z_1}(y_0)}{} = \frac{1-t}{1-t_0} \norm{y_0 - \P_{\Z_1}(y_0)}{} = 
  \frac{1-t}{1-t_0}(1-t_0)\rho = (1-t)\rho,
\]
therefore
\begin{equation}\label{y(t) in bordo(B_((1-t)delta)) for t > t_0}
  y(t) \in \partial(\Z_1 + \overline{B}_{(1-t)\rho}) \qquad \forall t \in \clint{t_0,1}.
\end{equation}
Therefore, since
\[
  y'(t) =
  \begin{cases}
    0 & \text{if $t \in \opint{0,t_0}$} \\
    \dfrac{1}{1-t_0}(\Proj_{\Z_2}(y_0) - y_0) & \text{if $t \in \opint{t_0,1}$}
  \end{cases},
\]
we infer from Lemma \ref{L:lemma tecnico-2}-(iv) that 
$-y'(t) \in N_{\G(t)}(y(t))$ for every $t \in \opint{0,1} \setmeno \{t_0\}$, and we are done. 
\end{proof}

Now we prove the main theorem of the paper. 

\begin{Thm}\label{main thm}
Assume that $r \in \opint{0,\infty}$ and let $\Conv_r(\H)$ be the family of $r$-prox-regular subsets of $\H$. 
Assume that $\C \in \BR^\r_\loc(\clsxint{0,\infty};\Conv_r(\H))$, i.e. that 
$\C : \clsxint{0,\infty} \function \Conv_r(\H)$ has locally bounded retraction. If $y_0 \in \C(0) + B_r(0)$ and 
$\e(\C(t-),\C(t)) < r$ for every $t \ge 0$, then there exists a unique 
$y \in \BV^\r_\loc(\clsxint{0,\infty};\H)$ such that there exists a Borel measure 
$\mu : \borel(\clsxint{0,\infty}) \function \clint{0,\infty}$ and a function $v \in \L^1_\loc(\mu;\H)$ such that
\begin{alignat}{3}
  & y(t) \in \C(t) & \qquad & \forall t \in \clsxint{0,\infty}, \label{constr. arbBV sweep} \\
  & \D y = v \mu, \label{Dy = vmu}\\
  & v(t) \in -N_{\C(t)}(y(t)) & \qquad & \text{for $\mu$-a.e. $t \in \clsxint{0,\infty}$}, \label{diff inclu}\\
  & y(0) = \Proj_{\C(0)}(y_0). \label{i.c. arbBV sweep}
\end{alignat}
Moreover 
\begin{equation}
\V(y,\clint{s,t})  \le  \ret(\C,\clint{s,t}) \qquad \forall t, s \ge 0,\ s < t,
\end{equation}
in particular 
\[
\cont(\C) \subseteq \cont(y).
\]
We will denote this unique solution by $\Sw(\C,y_0)$, so that if 
\[
  \textsl{D}_{\BR} := 
  \{(\C,y_0) \in \BR^\r_\loc(\clsxint{0,\infty};\Conv_r(\H)) \times \H\ :\ \ y_0 \in \C(0) + B_r(0)\}, 
\] 
then we can define $\Sw$ $:$ $\textsl{D}_{\BR} \function \BR_\loc(\clsxint{0,\infty};\H)$, 
the solution operator associating with every pair $(\C,y_0)$ the only $y$ satisfying 
\eqref{constr. arbBV sweep}--\eqref{i.c. arbBV sweep}.
\end{Thm}

\begin{proof}
We recall that 
$\ell_\C : \clsxint{0,\infty} \function \clsxint{0,\infty}$ is defined by
\begin{equation}\label{ellC}
  \ell_\C(t) := \ret(\C;\clint{0,t}), \qquad t \ge 0,
\end{equation}
and observe that $\ell_\C$ is right-continuous because of the right-continuity of $\C$. The function $\ell_\C$ is increasing, therefore $\ell_\C^{-1}(\tau)$ is always a (possibly degenerate) interval for every 
$\tau \in \ell_\C(\clsxint{0,\infty})$ and we can define 
$\Ctilde : \ell_\C(\clsxint{0,\infty}) \function \Conv_r(\H)$  
in the following way: 
\[
  \Ctilde(\tau) :=
    \begin{cases}
      \C(t)& \text{if $\ell_\C^{-1}(\tau)$ is a singleton and $\ell_\C^{-1}(\tau) = \{t\}$}, \\
      \C(t) & \text{if $\ell_\C^{-1}(\tau)$ is not a singleton and $\min\ell_\C^{-1}(\tau) = t$}.
    \end{cases}
\]
In the occurrence that $\ell_\C$ is not constant on any nondegenerate interval, then one has 
$\C(t) = \Ctilde(\ell_\C(t))$ for every $t \ge 0$ (this fact suggests that $\Ctilde$ can be considered as a sort of reparametrization of $\C$ by the ``arc length'' $\ell_\C$). Let us also observe that $\C$ is set-theoretically increasing in time on the intervals where $\ell_\C$ is constant (therefore the solution of the sweeping process driven by $\C$ is expected to be constant on these intervals). If $\sigma, \tau \in \ell_\C(\clsxint{0,\infty})$ with 
$\sigma < \tau$, then $s = \min\ell_\C^{-1}(\sigma) \le \sup\ell_\C^{-1}(\sigma) = s^*$ and 
$t = \min\ell_\C^{-1}(\tau) \le \sup\ell_\C^{-1}(\tau) = t^*$ for some $s, s^*, t, t^* \ge 0$, thus
$\e(\Ctilde(\sigma),\Ctilde(\tau))$ $= \e(\C(s),\C(t)) \le \ret(\C;\clint{s,t}) = \tau - \sigma$, and we have proved that 
$\Ctilde$ is $1$-Lipschitz continuous with respect to $\e$, i.e.
\[
  \e(\Ctilde(\sigma),\Ctilde(\tau)) \le \tau - \sigma \qquad \forall \tau, \sigma \in \ell_\C(\clsxint{0,\infty}),\ \tau \le \sigma. 
\] 
In order to extend the definition of $\Ctilde$ over the whole $\clsxint{0,\infty}$ we first set
\begin{equation}
  \rho_t := \e(\C(t-),\C(t)), \qquad t \in \discont(\C),
\end{equation}
and
\begin{equation}
  \A(t) := \C(t) + \overline{B}_{\rho_t}(0), \qquad t \in \discont(\C).
\end{equation}
By assumption we have that $0 < \rho_t < r$, moreover by virtue of Lemma \ref{L:lemma tecnico-2}-(iii) we have $\A(t)$ is a $(r - \rho_t)$-prox-regular set. Let us also observe that
\begin{equation}
  \C(t-) \subseteq \A(t)  \qquad \forall t \in \discont(\C).
\end{equation}
Now for any $t \in \discont(\C)$ we can define the curve $\G_t : \clint{0,1} \function \Conv_r(\H)$ by setting
\[
  \G_t := \G_{(\A(t),\C(t+)+\overline{B}_{\rho_t}(0))} : \clint{0,1} \function \Conv_r(\H), \qquad t \in \discont(\C),
\]
where $\G_{(\A(t),\C(t+)+\overline{B}_{\rho_t}(0))}$ is defined in Definition \ref{D:F}, so that
\[
  \G_t(\tau) = \C(\tau) + \overline{B}_{(1-\tau)\rho_t}(0) \qquad \forall \tau \in \clint{0,1}.
\]
We are now in position to define 
\begin{equation}\label{Ctilde}
  \Ctilde(\sigma) := \G_{t}\left(\frac{\sigma - \ell_\C(t-)}{\ell_\C(t) - \ell_\C(t-)}\right)  \qquad
       \text{if $\sigma \in \clsxint{\ell_\C(t-), \ell_\C(t)}$, if $\ell_\C(t-) \neq \ell_\C(t)$}.
\end{equation}
The resulting curve $\Ctilde : \clsxint{0,\infty} \function \Conv_r(\H)$ is a $1$-Lipschitz continuous function 
with respect to $\e$ because $\G_t$ is a geodesic connecting respectively $\A(t)$ to $\C(t)$, 
$\C(t-) \subseteq \A(t)$, and $\Ctilde$ is $1$-Lipschitz continuous on $\ell_\C(\clsxint{0,\infty})$. Therefore there exists a unique function $\yhat \in \Lip(\clsxint{0,\infty};\H)$ 
such that
\begin{alignat}{3}
  & \yhat(\sigma) \in \Ctilde(\sigma) & \quad & \forall \sigma \in \clsxint{0,\infty}, \label{yinC-hat} \\
  & \yhat'(\sigma) \in -N_{\Ctilde(\sigma)}(\yhat(\sigma)) & \quad & 
       \text{for $\leb^{1}$-a.e. $\sigma \in \clsxint{0,\infty}$}, 
     \label{diffincl-hat} \\
  & \yhat(0) = \P_{\Ctilde(0)}(y_{0}) = \P_{\C(0)}(y_{0}). &  \label{incond-hay}  
\end{alignat}
Now let us set
\begin{equation}\label{def-y}
   y := \Sw(\Ctilde, y_0) \circ \ell_\C = \yhat \circ \ell_{\C},
\end{equation}
and let us prove that $y$ solves \eqref{constr. arbBV sweep}-\eqref{i.c. arbBV sweep}. It is obvious that 
$y(t) \in \C(t)$ when $\ell_\C^{-1}(\ell_{\C}(t)) = \{t\}$. If instead $\ell_\C^{-1}(\ell_{\C}(t))$ is not a singleton, we have that $y(t) = \yhat(\ell_\C(t)) \in \Ctilde(\ell_\C(t)) = \C(\min\ell_\C^{-1}(\ell_\C(t))) \subseteq \C(t)$, because $\C$ is set-theoretically increasing on $\ell_\C^{-1}(\ell_{\C}(t))$. Thus condition \eqref{constr. arbBV sweep} is satisfied. The initial condition \eqref{i.c. arbBV sweep} is obviously satisfied. Since $\yhat$ is Lipschitz continuous and 
$\ell_\C$ is right-continuous and increasing, it is clear that $y \in \BV^\r_\loc(\clsxint{0,\infty};\H)$, so that 
$\discont(y) \subseteq \discont(\ell_\C) = \discont(\C)$. 
If $v : \clsxint{0,\infty} \function \H$ is defined by
\begin{equation}\label{density w}
  v(t) := 
  \begin{cases}
  \yhat' (\ell_{\C}(t)) & \text{if $t \in \cont(\ell_\C)$} \\
  \ \\
  \dfrac{\yhat(\ell_\C(t)) - \yhat(\ell_\C(t-))}{\ell_\C(t) - \ell_\C(t-)} & \text{if $t \in \discont(\ell_\C)$}
  \end{cases},
\end{equation}
then from Proposition \ref{P:BV chain rule}-(ii) we infer that $\D y = v \D\ell_\C$, i.e. \eqref{Dy = vmu} holds with 
$\mu = \D \ell_\C$. Let us set
\begin{equation}\label{Z}
  Z := \{t \in \clsxint{0,\infty}\ :\ -\yhat'(t) \not\in N_{\Ctilde(t)}(\yhat(t))\}.
\end{equation} 
From formula \eqref{diffincl-hat} we deduce that
\[
  \leb^{1}(Z) = 0,
\]
therefore, thanks to Proposition \ref{P:BV chain rule}-(i), we have that
\begin{align}
  & \D\ell_{\C}(\{t \in \cont(\ell_{\C})\ :\ -v(t) \not\in N_{\C(t)}(y(t))\}) \notag \\
  = & \D\ell_{\C}
         (\{t \in \cont(\ell_\C)\ :\ -\yhat'(\ell_{\C}(t)) \not\in N_{\Ctilde(\ell_{\C}(t))}(\yhat(\ell_{\C}(t))\})  
          \notag \\ 
  = & \D\ell_{\C}(\{t \in \cont(\ell_{\C})\ :\ \ell_{\C}(t) \in Z\}) = \leb^{1}(Z) = 0. \label{Dl_C(...)}
\end{align}
Now let us fix $t \in \discont(\ell_\C)$ and observe that 
\[
  \Ctilde(\sigma) = \G_{(\A(t),\C(t))}\left(\frac{\sigma-\ell_\C(t-)}{\ell_\C(t) - \ell_\C(t-)}\right) \qquad 
  \forall \sigma \in \opint{\ell_\C(t-), \ell_\C(t)},
\] 
thus by the semigroup property of \eqref{yinC-hat}--\eqref{incond-hay}, by Lemma 
\ref{L:particular sweeping process}, and thanks to the fact that 
\begin{equation}
  y(t-) = \yhat(\ell_\C(t-)) \in \Ctilde(\ell_\C(t-)) =  \A(t),
\end{equation} 
we have
\begin{align}
  -v(t) =
  \frac{\yhat(\ell_\C(t-)) - \yhat(\ell_\C(t))}{\ell_\C(t) - \ell_\C(t-)} =
  \frac{\yhat(\ell_\C(t-)) - \P_{\C(t)}(\yhat(\ell_\C(t-)))}{\ell_\C(t) - \ell_\C(t-)},
\end{align}
hence
\[
  -v(t) \in N_{\C(t)}(\P_{\C(t)}(\yhat(\ell_\C(t-)))) = N_{\C(t)}(\yhat(\ell_\C(t))) = N_{\C(t)}(y(t)) \qquad \forall t \in \discont(\C),
\]
that together with \eqref{Z}-\eqref{Dl_C(...)} implies that 
\[
  \D\ell_{\C}(\{t \in \clsxint{0,\infty}\ :\ -v(t) \not\in N_{\C(t)}(y(t))\}) = 0,
\]
hence \eqref{diff inclu} also holds with $\mu = \D \ell_\C$.

Concerning uniqueness, let us observe that since $\C$ has local bounded retraction, then 
$\discont(\C) = \discont(\ell_\C)$, thus $\discont(\C)$ is at most denumerable. Let us consider first the case when  the cardinality of $\discont(\C)$ is equal to the cardinality of $\en$. Since 
$\e(\C(t-),\C(t)) = \ell_\C(t) - \ell_\C(t-)$ for every $t \in \discont(\C)$, we infer that there exist $A_ -$ and $A_+$ such that $\discont(\C) = A_- \cup A_+$, $A_- \cap A_+ = \void$, $A_- = \{t_n\ :\ n \in \en\}$ with 
$\e(\C(t_n-),\C(t_n)) < r/2$ for all $n \in \en$, and $A_+= \{s_m\ :\ m=1,\ldots,N\}$ with 
$r/2 \le \e(\C(s_m-),\C(s_m)) < r$ for all $m = 1, \ldots, N$ and $s_0 := 0 < s_1 < s_2 < \cdots < s_N$. Arguing iteratively, thanks to the existence part of this proof, and to \cite[Proposition 3.6]{Thi16}, we have that for every 
$m \in \{1, \ldots, N\}$ there is only one solution $y_m \in \BV^\r(\clsxint{s_{m-1},s_m};\H)$ to the sweeping process 
\begin{alignat}{3}
  & y_m(t) \in \C(t) & \qquad & \forall t \in \clsxint{s_{m-1},s_m}, \label{constr. arbBV sweep-sm} \\
  & \D y_m = v_m \mu_m, \label{Dy = vmu-sm}\\
  & v_m(t) \in -N_{\C(t)}(y_m(t)) & \qquad & \text{for $\mu$-a.e. $t \in \clsxint{s_{m-1},s_m}$}, \label{diff inclu-sm}\\
  & y_m(s_{m-1}) = \P_{\C(s_{m-1})}(y_{m-1}(s_{m-1}-)), \label{i.c. arbBV sweep-sm}
\end{alignat}
because $y_{m-1}(s_{m-1}-) \in \A(s_{m-1})$. 
Moreover, thanks to the existence part of this proof, partitioning 
$\clsxint{s_{N},\infty}$ into a denumerable union of precompact intervals, and using \cite[Proposition 3.6]{Thi16}, we infer that there is only one solution $y_{N+1} \in \BV^\r_\loc(\clsxint{s_{N},\infty};\H)$ to the sweeping process 
\begin{alignat}{3}
  & y_{N+1}(t) \in \C(t) & \qquad & \forall t \in \clsxint{s_{N},\infty}, \label{constr. arbBV sweep-sN} \\
  & \D y_{N+1} = v_{N+1} \mu_{N+1}, \label{Dy = vmu-sN}\\
  & v_{N+1}(t) \in -N_{\C(t)}(y_{N+1}(t)) & \qquad & \text{for $\mu$-a.e. $t \in \clsxint{s_{N},\infty}$}, 
  \label{diff inclu-sN}\\
  & y_{N+1}(s_{N}) = \P_{\C(s_{N})} (y_{N}(s_{N}-)), \label{i.c. arbBV sweep-sN}
\end{alignat}
because $y_{N}(s_{N}-) \in \A(s_N)$. We claim that the function $y$ defined piecing together all the $y_m$'s, i.e.
\[
  y(t) := 
  \begin{cases}
  y_m(t) &  \text{if $t \in \clsxint{s_{m-1}, s_m}$, $m \in \{1,\ldots,N\}$}, \\
  y_{N+1}(t) &  \text{if $t \in \clsxint{s_{N}, \infty}$},
  \end{cases}
\] 
is the only solution to \eqref{constr. arbBV sweep}--\eqref{i.c. arbBV sweep}. Indeed setting 
\[
  \mu(B) := \mu_{N+1}(B \cap \clsxint{s_{N},\infty}) + \sum_{m=1}^{N+1}\mu_{m}(B \cap \clsxint{s_{m-1},s_m}),
  \qquad B \in \borel(\clsxint{0,\infty}), 
\]
and 
\[
  v(t) := 
  \begin{cases}
  v_m(t) &  \text{if $t \in \clsxint{s_{m-1}, s_m}$, $m \in \{1,\ldots,N\}$}, \\
  v_{N+1}(t) &  \text{if $t \in \clsxint{s_{N}, \infty}$},
  \end{cases}
  \qquad t \in \clsxint{0,\infty},
\]
then $\D y = w \mu$, and for every $m \in \{1,\ldots,N\}$, by \eqref{diff inclu}, we have
\[
  y(s_m) - y(s_m-) = \D y(\{s_m\}) = v_{m+1}(s_m)\mu_{m+1}(\{s_m\})
\]
and
\[
  y(s_m) - y(s_m-) = \P_{\C(s_{m})}(y_{m}(s_{m}-)) - y(s_m-) \in -N_{\C(s_m)}(y_{m}(s_m))
\]
so that, since $\mu_{m+1}(\{s_m\}) > 0$,
\begin{equation}\label{a f cond}
  v(s_m) \in -N_{\C(s_m)}(y(s_m)) \qquad \forall m \in \{1,\ldots,N\},
\end{equation}
and we know that on the jumps $t$ of $\C$ any solution $y$ satisfies a fortiori the condition \eqref{a f cond}, because its density with respect to a positive measure must belong to the $-N_{\C(t)}(y(t))$. The case when the cardinality of $\discont(\C)$ is finite follows the same argument and is actually simpler.

Finally from the definition of pointwise variation, from \eqref{def-y}, and from the $1$-Lipschitz continuity of $\yhat$, we infer that
\begin{align}
  \V(y,\clint{s,t}) = \V(\yhat, \clint{\ell_{\C}(s),\ell_\C(t)}) \le \V(\ell_\C, \clint{s,t}) = \ret(\C,\clint{s,t}),
\end{align}
whenever $s < t$, and the theorem is completely proved.
\end{proof}

\begin{Cor}
Assume that $r \in \opint{0,\infty}$, $\C \in \BR^\r_\loc(\clsxint{0,\infty};\Conv_r(\H))$, $y_0 \in \C(0) + B_r(0)$, and let 
$y := \Sw(\C,y_0)$ be the unique solution of the sweeping process 
\eqref{constr. arbBV sweep}--\eqref{i.c. arbBV sweep}.
\begin{itemize}
\item[(i)]
If $\C \in \Czero\BR_\loc(\clsxint{0,\infty};\Conv_r(\H))$ then 
$\Sw(\C,y_0) \in \Czero\BV_\loc(\clsxint{0,\infty};\H)$. 
\item[(ii)]
If $\C$ is locally absolutely continuous, in the sense that there exists a locally absolutely continuous function 
$\alpha : \clsxint{0,\infty} \function \mathbb{R}$ such that
\[
  \e(\C(t),\C(s)) \le \int_t^s \alpha'(\tau) \de \tau \qquad \forall t, s \ge 0, \ t < s,
\]
then
$y$ is locally absolutely continuous, i.e. $y \in W^{1,1}(\clint{0,T};\H)$ for every $T > 0$. In this case we have that \eqref{constr. arbBV sweep} and \eqref{i.c. arbBV sweep} hold together with 
\begin{equation}\label{fine}
  y'(t) \in -N_{\C(t)}(y(t)) \qquad \text{for $\leb^1$-a.e. $t \in \clsxint{0,\infty}$}.
\end{equation}
\item[(iii)]
If $\C$ is Lipschitz continuous in the sense that there exists $L \ge 0$ such that
\[
  \e(\C(t),\C(s)) \le L(s-t) \qquad \forall t, s \ge 0, \ t < s,
\]
then $y \in \Lip(\clsxint{0,\infty};\H)$ with $\Lipcost[y] \le L$. 
\end{itemize}
\end{Cor}

\begin{proof}
Recalling \eqref{ellC} and \eqref{Ctilde}, we have 
\[
  \Sw(\C, y_0) = \Sw(\Ctilde, y_0) \circ \ell_\C,
\]
with $\Lipcost[\Sw(\Ctilde, y_0)] \le 1$. Therefore statement (i) follows from the continuity of $\ell_\C$.
Statement (iii) and the first part of statement (ii) follow from the fact that the local absolute continuity of $\C$ (respectively the $L$-Lipschitz continuity of $\C$) is equivalent to the local absolute continuity of $\ell_\C$ (respectively the $L$-Lipschitz continuity of $\C$), see, e.g., \cite[Section 3d]{Mor74b}. We are left to prove \eqref{fine}. We have $\D y = y' \leb^1$, and if we set $\yhat := \Sw(\Ctilde) \in \Lip(\clint{0,T};\H)$ and
\begin{equation}\label{finefine}
  Z := \{s \in \clsxint{0,\infty}\ :\ -\yhat'(s) \not\in N_{\Ctilde(s)}(\yhat(s))\},
\end{equation}
then by Theorem \ref{Lip theorem}
\begin{equation}\label{Z has null measure - bis}
  \leb^{1}(Z) = 0.
\end{equation}
We have $y'(t) = 
  \ell_{\C}'(t) \yhat'(\ell_{\C}(t))$ for $\leb^{1}$-a.e. $t \ge 0$, therefore
\begin{align}\label{null set for y}
  & \sp \leb^{1}(\{t \in \opint{0,T}\ :\ -y'(t) \not\in N_{\C(t)}(y(t))\}) \notag \\
  & =    \leb^{1}(\{t \in \opint{0,T}\ :\ -\ell_{\C}'(t)\yhat'(\ell_{\C}(t)) 
              \not\in N_{\Ctilde(\ell_{\C}(t))}(\yhat(\ell_{\C}(t)))\}).
\end{align}  
Using  the fact that $N_{\mathcal{K}}(x)$ is a cone, we have
\begin{align}\label{null set for y - 2}
   & \sp \{t \in \opint{0,T}\ :\ -\ell_{\C}'(t)\yhat'(\ell_{\C}(t)) 
              \not\in N_{\Ctilde(\ell_{\C}(t))}(\yhat(\ell_{\C}(t)))\} \notag \\
   & =    \{t \in \opint{0,T}\ :\ -\yhat'(\ell_{\C}(t)) 
              \not\in N_{\Ctilde(\ell_{\C}(t))}(\yhat(\ell_{\C}(t))),\ \ell_{\C}'(t) \neq 0\} \notag \\
   & = \{t \in \opint{0,T}\ :\ \ell_{\C}(t) \in Z,\ \ell_{\C}'(t) \neq 0\}.
\end{align}
Moreover observe that (cf. \cite[Lemma 3.1]{Rec11b})
\begin{equation}\label{fine2}
  \leb^1(\{t \ge 0\ :\ \exists \ell'_\C(t) \neq 0,\ \ell_\C(t) \in Z\}) = 0,
\end{equation}
therefore we conclude by collecting equations 
\eqref{finefine}, \eqref{Z has null measure - bis}, \eqref{null set for y - 2} and \eqref{fine2}.
\end{proof}

\begin{Rem}
Part (iii) of the previous Corollary can be proved directly following the proof of Theorem \ref{Lip theorem}.
\end{Rem}


\section*{Acknowledgment}

I am very grateful to the referee for her/his remarks and corrections, and for pointing out that many parts of Lemma \ref{L:lemma tecnico-2} were already proved in \cite{Thi23b}.



\end{document}